\documentclass[12pt,a4paper]{amsart}

\usepackage[top=3cm, bottom=3cm, left=2.5cm, right=2.5cm]{geometry}
\usepackage{etex}
\usepackage{amssymb}
\usepackage{amscd}
\usepackage[latin1]{inputenc}
\usepackage{wrapfig}
\usepackage{psfrag}
\usepackage{epsfig}
\usepackage{epic}
\usepackage{eepic}
\usepackage{pifont}
\usepackage{mathrsfs}
\usepackage{enumerate}

\usepackage{hyperref}

\usepackage{caption}

\usepackage{setspace}

\usepackage{tikz}
\usetikzlibrary{shapes, quotes}

\numberwithin{figure}{section}

\numberwithin{equation}{section}
\swapnumbers

\newtheorem{Lemma}[equation]{Lemma}
\newtheorem{Theorem}[equation]{Theorem}
\newtheorem*{thm*}{Theorem}

\newtheorem*{Question*}{Question}
\newtheorem{Question}[equation]{Question}

\newtheorem{Corollary}[equation]{Corollary}
\newtheorem{Claim}[equation]{Claim}
\newtheorem*{Lemma*}{Lemma}
\newtheorem*{Corollary*}{Corollary}

\theoremstyle{remark}
\newtheorem{Remark}[equation]{Remark}

\theoremstyle{definition}

\newtheorem{Definition}[equation]{Definition}
\newtheorem{noTitle}[equation]{}

\newenvironment{enumeratei}{\begin{enumerate}%
                           [\upshape (i)]}%
                                       {\end{enumerate}}
\newcommand{\itemref}[1]{\eqref{#1}}

\renewcommand{\v} {{\lambda;\delta_1,\ldots,\delta_k}}

\newcommand{\R}{\mathbb{R}}
\newcommand{\Z}{\mathbb{Z}}
\newcommand{\C}{\mathbb{C}}

\newcommand{\CP}{{\mathbb C}{\mathbb P}}

\newcommand{\g} {\mathfrak{g}}

\newcommand{\Id}{\mathrm{Id}}

\newcommand{\tM}{\tilde{M}}

\newcommand{\tow}{\tilde{\omega}}

\newcommand{\la}{\langle}
\newcommand{\ra}{\rangle}

\newcommand{\mult} {\text{mult}}

\newcommand{\ow}{\omega}

\DeclareMathOperator{\grad}{grad}

\DeclareMathOperator{\Ham} {Ham}
\DeclareMathOperator{\Symp}{Symp}

\newif\ifdebug                                                      %
\debugfalse

\newcommand{\printname}[1]
   {\smash{\makebox[0pt]{\hspace{-1.0in}\raisebox{8pt}{\tiny #1}}}}

\newcommand{\labell}[1] {\label{#1}{\ifdebug{\printname{#1}}\fi}}

\newcommand{\mute}[1] {}
\begin{document}

\title[Cyclic actions need not extend to circle actions]
{Homologically trivial symplectic cyclic actions need not extend to Hamiltonian circle actions}

\author[River Chiang]{River Chiang}
\address{Department of Mathematics, National Cheng Kung University, Tainan 701, Taiwan}
\email{riverch@mail.ncku.edu.tw}

\author[Liat Kessler]{Liat Kessler}
\address{Department of Mathematics, Physics, and Computer Science, University of Haifa,
at Oranim, Tivon 36006, Israel}
\email{lkessler@math.haifa.ac.il}

\begin{abstract}
	We give examples of symplectic actions of a cyclic group, inducing a trivial action on homology, on four-manifolds that admit Hamiltonian circle actions, and show that they do not extend to Hamiltonian circle actions. Our work applies holomorphic methods to extend combinatorial tools developed for circle actions to study cyclic actions.
\end{abstract}


\maketitle

\section{Introduction}

\begin{Question} \labell{Q1}
Does every homologically trivial symplectic cyclic action extend to a Hamiltonian circle action on
closed connected symplectic four-manifolds that admit Hamiltonian circle actions?
\end{Question}

By a \emph{cyclic action} we mean an action of a cyclic group $\Z_n = \Z / n \Z$ of finite order $1<n< \infty$.
An action is called \emph{homologically trivial} if it induces the identity map on homology.
A circle action is  always homologically trivial because the circle group is connected.
An effective symplectic action of a Lie group $G$ on a symplectic manifold $(M,\omega)$ is called \emph{Hamiltonian} if it admits a \emph{moment map}, i.e., an equivariant smooth map $$\Phi \colon M \to {\g}^{*}$$ such that the components $\Phi^{\xi}=\langle \Phi,\xi \rangle$ satisfy Hamilton's equation
$$d \Phi^{\xi}=-\iota(\xi_{M})\omega,$$
for all $\xi \in \g$. Here $\xi_M$ is the vector field that generates the action on $M$ of the one-parameter subgroup $\{\exp(t \xi) \, |\, t \in \R\}$ of $G$.
There are four-manifolds that admit symplectic cyclic actions but not circle actions \cite{chen1, chen2}. Here we are only interested in the symplectic manifolds that admit Hamiltonian circle actions.

Question \ref{Q1} was asked by Jarek Kedra
at the ``Hofer 20" conference in Edinburgh in July 2010. It remained open since then.

In this paper we show that the answer is ``No''. The novelty of our work is in the application of holomorphic methods to extend combinatorial tools developed for circle actions to study cyclic actions.

In Section \ref{ex-simply} we give an example of a symplectic action of $\Z_{2}$, acting trivially on homology, on a symplectic manifold $M$ that is obtained from $\CP^2$ with the Fubini--Study form by six symplectic blowups,
and show that it cannot extend to a Hamiltonian circle action, though the manifold does admit a Hamiltonian circle action. Note that $H^{1}(M;\R)=\{0\}$, hence every symplectic circle action on $M$ is Hamiltonian.

In the construction of the $\Z_{2}$-action, we start from a Hamiltonian circle action on $\CP^2$ blown up symplectically five times, that admits a $\Z_{2}$-fixed sphere, and perform a $\Z_{2}$-equivariant complex blowup at a point in the $\Z_{2}$-fixed sphere that is not $S^1$-fixed. We use Nakai's criterion for the existence of a K\"ahler form in a cohomology class, as well as the fact that the obtained complex manifold is a weak del Pezzo surface of degree $3$, and the results of Bruce--Wall \cite{BW} on holomorphic curves in cubic surfaces.
We then apply the results of  Li--Liu \cite{liliu-ruled}, Taubes--Seiberg--Witten \cite{taubes1} and Karshon--Kessler \cite{blowups} to show that this K\"ahler form is obtained by symplectic blowups, of certain sizes, from the  Fubini--Study form on $\CP^2$.

 In the proof that the action does not extend to a Hamiltonian circle action, we show that the configuration of invariant spheres of negative self intersection in the constructed $\Z_2$-action is different from that of
any Hamiltonian circle action on the symplectic manifold.
For this we use the decorated graphs associated to Hamiltonian circle actions on four-manifolds by Karshon \cite{karshon} and  Karshon--Kessler--Pinsonnault \cite{algann} characterization of circle actions on symplectic blowups of $\CP^2$.

In Section \ref{ex-nonsimply} we give an example of a homologically trivial symplectic action of $\Z_{2}$ on a symplectic manifold that is obtained by three symplectic blowups from $\Sigma \times S^2$ with a symplectic ruling, where $\Sigma$ is a Riemann surface of positive genus. We show that it cannot extend to a Hamiltonian circle action, though the manifold does admit a Hamiltonian circle action. In this section we use Holm--Kessler \cite{holmkessler} characterization of Hamiltonian circle actions on symplectic four-manifolds that are not simply connected.

We conclude our paper, in Section \ref{remarks}, with remarks concerning  
the generator of the $\Z_2$-action from Section \ref{ex-simply}. 

\bigskip

{\bf Acknowledgements.}  We thank Yael Karshon and Leonid Polterovich for telling us about Kedra's question. We thank them as well as Masanori Kobayashi and Kaoru Ono for helpful suggestions and stimulating discussions. We thank Shin-Yao Jow for teaching us how to do Claim \ref{generators} \itemref{part3}.
RC is supported by MoST 103-2115-M-006-004 and NCTS, Taiwan.

\bigskip

\section{Hamiltonian circle actions and decorated graphs}

A \emph{Hamiltonian $S^1$-space} is a triple $(M, \ow, \Phi)$ that consists of a closed connected symplectic four-manifold $(M, \ow)$ with a Hamiltonian $S^1$-action and moment map $\Phi: M \to \R$, where we identify the Lie algebra of $S^1$ and its dual with $\R$.
The moment map
$\Phi$ is a Morse--Bott function whose
critical set corresponds to the fixed point set of the action, which is a union of symplectic submanifolds by the local normal form theorem for Hamiltonian actions.
Since $M$ is four-dimensional, the critical set contains only
isolated points and two-dimensional submanifolds. The latter
can only occur at the extrema of $\Phi$.  The maximum and minimum of the moment map is each attained on exactly one component of the fixed point set.

An $S^1$ invariant Riemannian metric $\langle \cdot, \cdot \rangle$ on a Hamiltonian $S^1$-space $(M, \ow, \Phi)$ is called \emph{compatible} if the automorphism $J \colon TM \to TM$ defined by $\langle \cdot, \cdot \rangle=\omega(\cdot,J\cdot)$ is an almost complex structure, i.e., $J^2 = -\Id$. Such a $J$ is $S^1$ invariant, and we refer to it as a \emph{compatible almost complex structure}.
With respect to a compatible metric, the gradient vector field of the moment map is
\begin{equation}\labell{gradient}
\grad \Phi = -J \xi_M,
\end{equation}
where $J$ is the corresponding almost complex structure and $\xi_M$ is the vector field that generates the $S^1$-action.
The vector fields $\xi_M$ and $J\xi_M$ generate a $\C^\times = (S^1)^\C$ action.
The closure of a non-trivial $\C^\times$ orbit is a sphere, called a \emph{gradient sphere}. On a gradient sphere, $S^1$ acts by rotation with two fixed points at the north and south poles; all other points on the sphere have the same stabilizer.
A gradient sphere is smooth at its poles except in the following situation \cite{ah}:
if the gradient sphere is free (i.e., of trivial stabilizer) and the pole in question is an isolated minimum (or maximum) of $\Phi$ with both isotropy weights $> 1$ (or $< -1$).
In particular, a non-free gradient sphere is smoothly embedded.

The existence of non-free gradient spheres does not depend on the compatible metrics (or almost complex structures); each of these coincides with a $\Z_n$-sphere for some $n>1$, which is a connected component of the closure of the set of points in $M$ with stabilizer $\Z_n$.
On the other hand, a free gradient sphere connecting two interior fixed points can break into two free gradient spheres connecting to the maximum and minimum respectively, after a perturbation of compatible metrics in its neighborhood. This is demonstrated in Figure \ref{metric}, where the free gradient spheres are depicted as edges of label $1$ (as explained in the following paragraph).
For a \emph{generic} compatible metric on $(M,\omega, \Phi)$, there exists no free gradient sphere whose north and south poles are both interior fixed points \cite[Corollary 3.8]{karshon}.


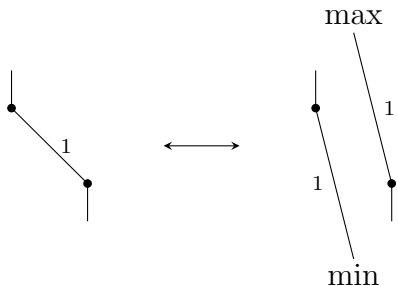
\begin{figure}[h]
\begin{tikzpicture}
[
	>=stealth,
	point/.style = {draw, circle,  fill = black, inner sep = 1pt},
]
	\draw node at (0,1) [point]{};
	\draw node at (1,0) [point]{};
	\draw (0,1.5) -- (0,1) -- node[right]{\tiny $1$}(1,0) -- (1,-0.5);	
	\draw[<->] (2,0.5) to (3,0.5);
	\draw node at (4,1) [point]{};
	\draw node at (5,0) [point]{};
	\draw node at (4.5,2.2) {max};
	\draw node at (4.5,-1.2) {min};
	\draw (4,1.5) -- (4,1) -- node[left]{\tiny $1$}(4.5,-1);	
	\draw (4.5,2) -- node[right]{\tiny $1$} (5,0) -- (5,-0.5);	
\end{tikzpicture}
\caption{Effect on the free gradient spheres after a change of compatible metric (or almost complex structure).}\labell{metric}
\end{figure}


With each Hamiltonian $S^1$-space, we associate a \textbf{decorated graph} as follows.
For each isolated fixed point, there is a vertex
labeled by its moment map value.
For each fixed surface $S$, there is
a \emph{fat} vertex labeled by its moment map value,
its size $\frac{1}{2\pi} \int_S \omega$,
and its genus $g$. The moment map value determines the vertical placement of a (fat or not) vertex. The horizontal placement is done at liberty.
For each $\Z_n$-sphere, $n>1$, there is an edge connecting the vertices corresponding to its fixed points and labeled by the integer $n$.
We add edges labeled $1$ for free gradient spheres so that any vertex corresponding to an interior fixed point is attached to one edge from above and one edge from below. We may add a number of \emph{redundant} edges labeled $1$ for free gradient spheres from the maximum of $\Phi$ to the minimum of $\Phi$ if $1$ is an isotropy weight (subject to sign) for both the maximum and minimum.

Two Hamiltonian $S^1$-spaces $(M,\omega,\Phi)$ and $(M',\omega',\Phi')$ are isomorphic if and only if they have the same decorated graph without redundant edges for a generic compatible metric \cite[Proposition 4.1]{karshon}. Vertical translations of the graph correspond to equivariant symplectomorphisms of the associated manifold, and flips upside down correspond to automorphisms of the circle.

\begin{Remark}
	The decorated graph defined here is the \emph{extended graph} in \cite[Section 5]{karshon}, which adds edges of label $1$ to the \emph{graph} in \cite[Section 2.1]{karshon}. For a generic compatible metric, the extended graph without redundant edges is determined by the graph \cite[Lemma 3.9]{karshon}.
\end{Remark}

\begin{Remark}
Since the Hamiltonian $S^1$-action is required to be effective, only edges of label $1$ can emerge from a fixed surface. Two edges sharing the same vertex have relatively prime edge labels since they represent the isotropy weights at the fixed point (subject to sign). Moreover,
the sphere $S$ corresponding to an edge of label $n$ is a symplectic sphere whose size $\frac{1}{2\pi} \int_S \ow$ is $1/n$ of the difference of the moment map values of its vertices.
\end{Remark}

\begin{Remark}
In our figures, we often omit some of the labels.
We omit the moment map labels; they are represented by the heights of the vertices.
We omit the genus labels:
in Sections \ref{background1} and \ref{ex-simply}, the genus labels are $0$; in Section \ref{ex-nonsimply}, the genus labels $g=g(\Sigma)$ are positive. 
We omit the size labels; the width of fat vertices,
or any horizontal space of the graphs, does not carry actual meanings.
We omit the edge label $1$ for free gradient spheres.
\end{Remark}

A \emph{(parametrized) $J$-holomorphic curve}  is a map from a compact Riemann surface to an almost complex manifold,
$f\colon (\Sigma, j) \to (M,J)$,  that satisfies
the Cauchy--Riemann equation
$$df \circ j = J \circ df.$$ When $\Sigma = \CP^1$, the map $f$ is called a \emph{(parametrized) $J$-holomorphic sphere}.
By \eqref{gradient} and the fact that a non-free gradient sphere is smoothly embedded, we can deduce that each $\Z_n$-sphere is $J$-holomorphic. This is a general phenomenon for embedded non-free spheres by the following lemma:

\begin{Lemma} \labell{clsta}
Let $(M,J)$ be an almost complex manifold such that $J$ is invariant under an action of a compact Lie group $G$ on $M$, i.e., $d \sigma_{g} \circ J=J \circ d \sigma_{g}$ for all $g \in G$. Let $S$ be an embedded sphere in $M$. Assume that $S$ is a connected component of the fixed point set of a non-trivial subgroup $\{e\} \neq H<G$. Then $S$ is a $J$-holomorphic sphere.
\end{Lemma}

\begin{proof}
First, it follows from the local normal form of a compact Lie group action on a manifold that
$TS=\{v \in TM \,|\, d\sigma_{h}v=v \text{ for all }h \in H\}$.
Since each of the maps $\sigma_{h}$ is $J$-holomorphic, for every $w \in TS$, the vector $Jw$ satisfies
$$d\sigma_{h}(Jw)=J(d\sigma_{h}w)=Jw.$$
Hence $TS$ is closed under $J$.
As an almost complex manifold of real dimension two, $(S,J|_{TS})$ is a complex manifold, and the embedding of the sphere is $J$-holomorphic (with respect to this complex structure).
\end{proof}

\subsection*{Equivariant symplectic blowups}

Recall that we can think of a symplectic blowup of size $\delta=r^2/2$ as cutting out an embedded ball of radius $r$ and identifying the boundary to an exceptional sphere via the Hopf map. This carries a symplectic form that integrates on the sphere to $2\pi \delta$.
For more details see \cite{GS:birational} and \cite[Section 7.1]{MS:intro}. If the embedding of the ball is $G$-equivariant centered at a $G$-fixed point, then the $G$-action extends to the symplectic blowup, see details in \cite{kk}. If the action is Hamiltonian, its moment map naturally extends to the equivariant symplectic blowup.

For all the possible effects of $S^1$-equivariant symplectic blowups at fixed points of a
Hamiltonian $S^1$-space $(M,\omega,\Phi)$, see Figures \ref{circlebl1}, \ref{circlebl2}, and \ref{circlebl3}, together with Figures \ref{circlebl2} and \ref{circlebl3} turned upside down \cite{karshon}.

\begin{Remark}\labell{blowup-size}
By  
\cite[Proposition 7.2]{karshon},
a Hamiltonian $S^1$-space
admits an $S^1$-equivariant blowup of size $\delta > 0$ centered at
some fixed point
if and only if one obtains a \emph{valid} decorated graph after the blowup. That is, the (fat or not) vertices created in the blowup do not surpass the other pre-existing (fat or not) vertices in the same chain of edges, and the fat vertices after the blowup have positive size labels. For quantitative description, see \cite[Lemma 3.3]{kk}.
\end{Remark}

\begin{figure}[h]
\begin{tikzpicture}
[
	>=stealth,
	point/.style = {draw, circle,  fill = black, inner sep = 1pt},
]
	\draw (-.5,1) -- node [right] {\tiny $m$} (0,0) node [left] {\tiny $\alpha$}
	-- node [right] {\tiny $n$} (-.5,-1);	
	\draw node at (0,0) [point]{};
	\draw[->] (1,0) to (2,0);
	\draw (3.5,1) -- node [right] {\tiny $m$} (4,0.3)  node [left] {\tiny $\alpha + m\delta$}
	-- node [right] {\tiny $m+n$} (4, -0.3) node [left] {\tiny $\alpha -n\delta$}
	-- node [right] {\tiny $n$} (3.5,-1);	
	\draw node at (4,0.3) [point]{};
	\draw node at (4, -0.3) [point]{};
\end{tikzpicture}
\caption{$S^1$-equivariant blowup at an interior fixed point.}\labell{circlebl1}
\end{figure}
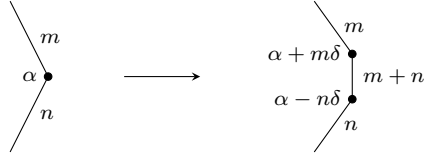

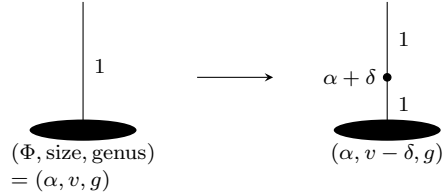
\begin{figure}[h]
\begin{tikzpicture}
[
	>=stealth,
	point/.style = {draw, circle,  fill = black, inner sep = 1pt},
]
	\fill (0,-0.7) ellipse [x radius=20pt, y radius=4pt] node [below=2pt, align=left, node font=\tiny] {$(\Phi, \text{size}, \text{genus})$\\$=(\alpha, v, g)$};
	\draw (0,1) -- node[right]{\tiny $1$}(0, -0.7);
	\draw[->] (1.5,0) to (2.5,0);
	\fill (4,-0.7) ellipse [x radius=20pt, y radius=4pt] node [below=2pt, align=left, node font=\tiny] {$(\alpha, v-\delta, g)$};
	\draw node at (4,0) [point]{} ;
	\draw (4,1) -- node[right]{\tiny $1$}(4, 0) node[left]{\tiny $\alpha + \delta$} -- node[right]{\tiny $1$} (4, -0.7);
\end{tikzpicture}
\caption{$S^1$-equivariant blowup at a point on the minimal surface.}\labell{circlebl2}
\end{figure}

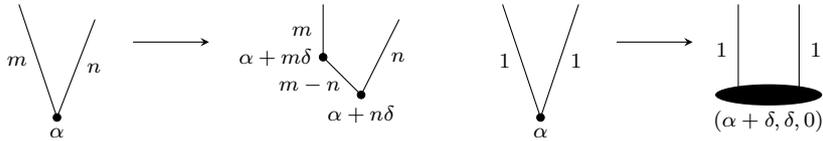
\begin{figure}[h]
\begin{tikzpicture}
[
	>=stealth,
	point/.style = {draw, circle,  fill = black, inner sep = 1pt},
]
	\draw (-0.5,0.5) -- node [left] {\tiny $m$} (0,-1) node [below] {\tiny $\alpha$}
	-- node [right] {\tiny $n$} (0.5,0.3);	
	\draw node at (0,-1) [point]{};
	\draw[->] (1,0) to (2,0);
	\draw (3.5,0.5) -- node [left] {\tiny $m$} (3.5,-0.2) node [left] {\tiny $\alpha + m\delta$}
	-- node [left, near end] {\tiny $m-n$} (4,-0.7) node [below] {\tiny $\alpha + n \delta$}
	-- node [right] {\tiny $n$} (4.5,0.3);	
	\draw node at (4,-0.7) [point]{};
	\draw node at (3.5,-0.2) [point]{};	
\end{tikzpicture}
\qquad
\begin{tikzpicture}	
[
	>=stealth,
	point/.style = {draw, circle,  fill = black, inner sep = 1pt},
]	
	\draw (5.5,0.5) -- node [left] {\tiny $1$} (6,-1) node [below] {\tiny $\alpha$}
	-- node [right] {\tiny $1$} (6.5,0.5);	
	\draw node at (6,-1) [point]{};
	\draw[->] (7,0) to (8,0);
	\fill (9,-0.7) ellipse [x radius=20pt, y radius=4pt] node [below=2pt] {\tiny $(\alpha + \delta, \delta, 0)$};
	\draw (8.6, -0.7) -- node [left]{\tiny $1$} (8.6, 0.5);
	\draw (9.4, -0.7) -- node [right]{\tiny $1$} (9.4, 0.5);
\end{tikzpicture}
\caption{$S^1$-equivariant blowup at an isolated minimum, where $m > n$.}\labell{circlebl3}
\end{figure}

\section{Blowups of $\CP^2$}
\labell{background1}

\subsection*{Symplectic blowups of $\CP^2$.}

For each nonnegative integer $k$, let $M_k$ denote a complex blowup of the complex projective plane $\CP^2$ at $k$ points.
We have a decomposition
\begin{equation} \labell{eqsplit}
H_2(M_k) = \Z L \oplus \Z E_1 \oplus \ldots \oplus \Z E_k,
\end{equation}
where $L$ is the image of the homology class of a line $\CP^1$
in $\CP^2$ under the inclusion map $H_2(\CP^2) \to H_2(M_k)$
and $E_1,\ldots,E_k$ are the homology classes of the exceptional
divisors.
As a smooth oriented manifold, $M_k$ is  $\CP^2 \# k \overline{\CP^2}$, the manifold that is obtained form $\CP^2$ (with the standard orientation) by $k$ iterations of connected sum with   $\overline{\CP^2}$ (i.e., $\CP^2$ with the opposite orientation).

\begin{Definition}\labell{bupdef}
A \textbf{blowup form}
on $M_k$ is a symplectic form for which
there exist pairwise disjoint
embedded symplectic spheres in the classes $L,E_1,\ldots,E_k$.
\end{Definition}

\begin{Lemma} \labell{deformation class} \cite[Lemmas 1.5, 1.6]{blowups}
The set of blowup forms on $M_k$ is an equivalence class
under the following equivalence relation:
symplectic forms $\omega$ and $\omega'$ on $M_k$
are equivalent if and only if there exists a diffeomorphism $f \colon M_k \to M_k$
that acts trivially on homology
and such that $f^*\omega$ and $\omega'$
are homotopic through symplectic forms.

Any two cohomologous blowup forms on $M_k$
are diffeomorphic by a diffeomorphism that acts trivially
on homology.
\end{Lemma}

This is a result of the work of Gromov and McDuff,
 \cite[2.4.A', 2.4.A1']{gromovcurves},
 \cite{rational-ruled}, \cite[Proposition~7.21]{MS:intro}, and \cite{isotopy}. 

For a blowup form $\omega$ on $M_k$, the vector $(\frac{1}{2\pi} \la [\omega] , L \ra;\frac{1}{2\pi} \la [\omega] , E_1 \ra,\ldots,\frac{1}{2\pi} \la [\omega] , E_k \ra )$ defined by the pairing of its cohomology class with the homology classes $(L;E_1,\ldots,E_k)$ equals the vector  $(\frac{1}{2\pi}\int_{C}\omega;\frac{1}{2\pi}\int_{C_1}\omega,\ldots,\frac{1}{2\pi}\int_{C_k}\omega)$, where $C, \, C_1,\ldots,C_k$ are embedded symplectic spheres in $L,\, E_1,\ldots,E_{k}$.
We will say that the blow up form (or its cohomology class) is \emph{encoded} by this vector.
By Lemma~\ref{deformation class},
the blowup form on $M_k$ encoded by the vector $(\v)$ is unique up to a diffeomorphism that acts trivially on homology.
We denote any of these symplectic manifolds by
$$ ( M_k , \omega_{\lambda;\delta_1,\ldots,\delta_k}).$$

By Li--Liu \cite[Theorem A]{liliu-ruled} and Taubes--Seiberg--Witten theory \cite{taubes1}, for every symplectic form on $M_k$ with a \emph{standard canonical class}, that is, for which  the pairing of the first Chern class $c_{1}(TM_k)$ with $(L;E_1,\ldots,E_k)$  equals $2\pi (3;1,\ldots,1)$, the classes $L,E_1,\ldots,E_k$ are represented by symplectically embedded spheres.
The next corollary follows. See \cite[Remark 6.1]{blowups}.
See also the discussion in Salamon \cite[Examples 3.8, 3.9, 3.10]{salamon}.

\begin{Corollary} \labell{cord}{\cite[Remark 6.1]{blowups}}
Every symplectic form on $M_k$ with a standard canonical class is a blowup form.
\end{Corollary}

\begin{Definition} \labell{reduced form}
Let $k \geq 3$, and let $\lambda, \delta_1, \ldots, \delta_k$ be real numbers.
The vector $(\lambda ; \delta_1 , \ldots , \delta_k)$
is \textbf{reduced} if
\begin{equation} \labell{conditions-1}
 \delta_1 \geq \ldots \geq \delta_k \quad \text{ and } \quad
   \delta_1 + \delta_2 + \delta_3 \leq \lambda.
\end{equation}
\end{Definition}

In \cite[Theorem 1.9]{blowups}, Karshon and Kessler show, using the work of Li--Li \cite{Li-Li02}, that for a vector $(\v)$ with positive entries that is reduced
and that satisfies the volume inequality
$\lambda^2 - {\delta_1}^2 - \ldots - {\delta_k}^2 > 0$, there exists a blowup form $\omega_{\v}$.

\subsection*{Hamiltonian $S^1$-actions on blowups of $\CP^2$.}

The decorated graphs, or their upside-down version, of Hamiltonian $S^1$-actions on $M_1$, namely, $\CP^2$ blown up at a point, with a blowup form $\ow_{\lambda;\delta_1}$ are shown in Figures \ref{graph-hirz-1} and \ref{graph-hirz-2}; the integer $n = 2\ell -1$ with $1 \leq \ell < \lambda/(\lambda-\delta_1)$.
We also mark the homology classes that correspond to the edges.


\begin{figure}[h]
\begin{tikzpicture}
[
	point/.style = {draw, circle,  fill = black, inner sep = 1pt},
	every edge quotes/.style = {font=\tiny, color = black, auto=left, inner sep = 1pt},
]
	\draw node at (0,1.5) [point]{};
	\draw node at (1.5,3) [point]{};
	\draw node at (2.25,0.75) [point]{};
	\draw node at (1.5,0) [point]{};
	\draw
	(0,1.5) edge ["$\ell L-(\ell -1)E_1$", "$c$"'] (1.5,3)
	(1.5,3) edge ["$L-E_1$", "$nc+d$"'] (2.25,0.75)
	(2.25,0.75) edge ["$(-\ell+1)L + \ell E_1$", "$c$"'] (1.5,0)
	(1.5,0) edge ["$L-E_1$", "$d$"'] (0,1.5); 	
\end{tikzpicture}
\begin{tikzpicture}
[
	point/.style = {draw, circle,  fill = black, inner sep = 1pt},
	every edge quotes/.style = {font=\tiny, color = black, auto=left, inner sep = 1pt},
]	
	\draw node at (0,0.75) [point]{};
	\draw node at (0,2.25) [point]{};
	\draw node at (1.5,3) [point]{};
	\draw node at (1.5,0) [point]{};
	\draw
	(0,0.75) edge ["$(-\ell+1)L + \ell E_1$", "$c$"'] (0,2.25)
	(0,2.25) edge ["$L-E_1$", "$nc-d$"' near start] (1.5,3)
	(1.5,3) edge ["$\ell L-(\ell -1)E_1$", "$c$"'] (1.5,0)
	(1.5,0) edge ["$L-E_1$", "$d$"'] (0,0.75);
\end{tikzpicture}
\caption{Decorated graphs for $(M_1, \ow_{\lambda;\delta_1})$ with only isolated fixed points, where $n=2\ell -1$ and $1 \leq \ell < \lambda / (\lambda-\delta_1)$.}\labell{graph-hirz-1}
\end{figure}
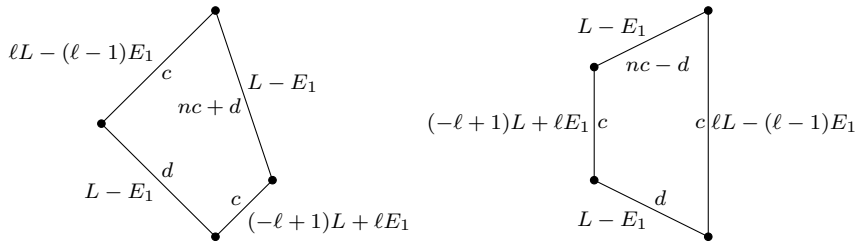

\begin{figure}[h]
\begin{tikzpicture}
[
	point/.style = {draw, circle,  fill = black, inner sep = 1pt},
	every edge quotes/.style = {font=\tiny, color = black, auto=left, inner sep = 1pt},
]		
	\draw node at (-0.5,1) [point]{};
	\draw node at (0.5,2) [point]{};
	\fill (0,0) ellipse [x radius=20pt, y radius=4pt] node[below=2pt]{\tiny $L-E_1$};
	\draw
	(-0.5,1) edge ["$n$"', "$L-E_1$"] (0.5, 2);
	\draw
	(-0.5,0) edge ["$(-\ell+1)L + \ell E_1$"] (-0.5,1);
	\draw
	(0.5,0) edge ["$\ell L-(\ell -1)E_1$"'] (0.5,2);
\end{tikzpicture}
\quad
\begin{tikzpicture}
[
	point/.style = {draw, circle,  fill = black, inner sep = 1pt},
	every edge quotes/.style = {font=\tiny, color = black, auto=left, inner sep = 1pt},
]		 	
	\fill (0,0) ellipse [x radius=20pt, y radius=4pt] node[below=2pt]{\tiny $\ell L-(\ell -1)E_1$};
	\fill (0,2) ellipse [x radius=20pt, y radius=4pt] node[above=2pt]{\tiny $(-\ell+1)L + \ell E_1$};	
	\draw (-0.5,0) edge ["$L-E_1$"] (-0.5, 2);
	\draw (0.5,0) edge ["$L-E_1$"'] (0.5, 2);
\end{tikzpicture}
\caption{Decorated graphs for $(M_1, \ow_{\lambda;\delta_1})$ with $1$ or $2$ fixed surfaces, where $n=2\ell-1$ and $1 \leq \ell < \lambda /(\lambda-\delta_1)$.}\labell{graph-hirz-2}
\end{figure}
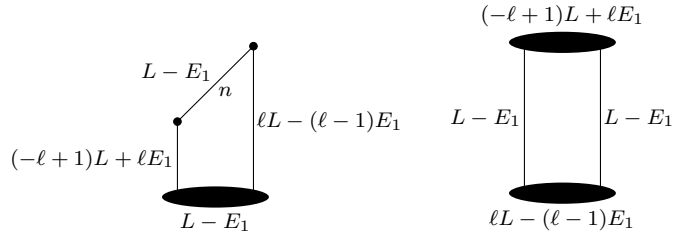

\begin{Remark}
Every Hamiltonian circle action on $\CP^2$ blown up once extends to a Hamiltonian torus action \cite[Theorem 1]{YK:max_tori}. Hence one can recover the decorated graphs in Figures \ref{graph-hirz-1} and \ref{graph-hirz-2} from the moment polytopes of toric actions, standard Hirzebruch trapezoids in this case, by projection along lines with a rational slope.
\end{Remark}

Karshon--Kessler--Pinsonnault \cite{algann} provide a characterization of Hamiltonian circle actions on blowups of $\CP^2$ as follows:

\begin{Theorem}\labell{character}\cite[Theorem 1.2]{algann} For $k \geq 3$, every Hamiltonian circle action on $M_k$ with a blowup form encoded by a reduced vector $(\lambda;\delta_1,\ldots,\delta_k)$ is obtained from a Hamiltonian circle action on $(M_{1},\omega_{\lambda;\delta_1})$ by a sequence of $S^1$-equivariant symplectic blowups of sizes $\delta_2,\ldots,\delta_k$, in this order.
\end{Theorem}


\subsection*{Weak del Pezzo surfaces}

A weak del Pezzo surface $M$ is a smooth projective surface whose anti-canonical class $-K_M$ is nef and big. Its degree is the number of the self-intersection of the anti-canonical class.
Any weak del Pezzo surface $M$ is isomorphic to $\CP^1 \times \CP^1$, the Hirzebruch surface $\mathbb{F}_2$ (both of degree $8$), or a blowup of $\CP^2$ at $0 \leq k \leq 8$ points in \emph{almost general position} (of degree $9-k$), that is,
\[
M = M_k \to M_{k-1} \to \dots \to M_1 \to M_0 = \CP^2
\] where each $\pi: M_i \to M_{i-1}$ is the blowup at a point $p_i \in M_{i-1}$ not lying on a $(-2)$-curve on $M_{i-1}$.
For details, see \cite[Section 8]{dolgachev}, or \cite{an},\cite{coray-tsfasman},\cite[Section III]{demazure}
(where they are called Gorenstein log del Pezzo surfaces, generalized del Pezzo surfaces, rational surfaces of negative type, respectively).


\section{Blowups of $\Sigma \times S^2$}


\subsection*{Symplectic blowups of $\Sigma \times S^2$.}

Let $(\Sigma,j)$ be a compact connected  Riemann surface of positive genus $g=g(\Sigma)>0$ endowed with a complex structure $j$.  
We fix a smooth structure on the trivial bundle 
$\Sigma \times S^2$, and equip it
with the product complex structure, in which each fiber is a holomorphic sphere. 
We fix basepoints $\ast\in S^2$ and $\ast\in \Sigma$, and denote 
$F:=[\ast \times S^2], \,\,\, B:=[\Sigma \times \ast],$
the homology classes in $H_2(\Sigma \times S^2;\Z)$. 

For a non-negative integer $k$, let $W_{k}$ denote 
a complex blowup of $\Sigma \times S^2$ at $k$ points.
Let $E_1, \ldots, E_k$
denote the homology classes of the exceptional divisors. 

We say that a vector $(\lambda_F,\lambda_B;\delta_1,\ldots,\delta_k)$ in $\R^{2+k}$
{\bf encodes} a cohomology class $\Omega \in H^2(W_{k};\R)$
if 
$\frac{1}{2\pi} \left< \Omega , F \right> = \lambda_F$,  
$\frac{1}{2\pi} \left< \Omega , E_j \right> = \delta_j$
for $j= 1, \ldots, k$, 
and $\frac{1}{2\pi} \left< \Omega , B \right> = \lambda_B$.

For $k \geq 2$, 
we say that a cohomology class $\Omega \in H^2(W_{k};\R)$ encoded by a vector  $(\lambda_F,\lambda_B;\delta_1,\ldots,\delta_k)$ is {\bf reduced}
if 
$ \delta_1 \geq \dots \geq \delta_k$ and $\delta_1 + \delta_2 \leq \lambda_F$.

\begin{Definition}\label{def:blowup}
A {\bf blowup form} on $W_{k}$ is 
a symplectic form for which there exist disjoint 
embedded symplectic spheres (oriented by the symplectic form)  in the 
homology classes $F,E_1, \ldots, E_k$.
\end{Definition}

By \cite{isotopy}, the blowup form on $W_{k}$ (whose cohomology class is) encoded by the vector $(\lambda_F,\lambda_B;\delta_1,\ldots,\delta_k)$ is unique up to a diffeomorphism that acts trivially on homology. We denote any of these symplectic manifolds by $$(W_{k},\omega_{\lambda_F,\lambda_B;\delta_1,\ldots,\delta_k}).$$

\subsection*{Hamiltonian $S^1$-actions on blowups of $\Sigma \times S^2$.}
\label{actionsruled} 

By \cite[Theorem 6.3 and Lemma 6.15]{karshon},
the decorated graph of a Hamiltonian circle action on $\Sigma \times S^2$ with 
$\omega_{\lambda_F,\lambda_B}$
consists of two fat vertices with the same genus label $g$; moreover the difference of their heights is $\lambda_F$ and their area labels are $\{r,r+ns\}$ such that $r >0$, the integer $n=2\ell$ with $0\leq \ell < \lambda_B/ \lambda_F$, and $r+\ell \lambda_F=\lambda_B$.
See also \cite[Proposition 3.2]{holmkessler}. The decorated graphs with a redundant edge are shown in Figure \ref{graph-ruled}. We also mark the homology classes that correspond to the edges. 

\begin{figure}[h]
\begin{tikzpicture}
[
	>=stealth,
	point/.style = {draw, circle,  fill = black, inner sep = 1pt},
	every edge quotes/.style = {font=\tiny, color = black, auto=left, inner sep = 1pt},
]	
	\fill (0,0) ellipse [x radius=20pt, y radius=2pt] node[below=2pt]{\color{black}\tiny $B-\ell F$};
	\fill (0,2) ellipse [x radius=20pt, y radius=2pt] node[above=2pt]{\color{black}\tiny $B+\ell F$};
	\draw
	(-0.5,0) edge ["$F$"] (-0.5, 2);
\end{tikzpicture}
\hspace{2cm}
\begin{tikzpicture}
[
	>=stealth,
	point/.style = {draw, circle,  fill = black, inner sep = 1pt},
	every edge quotes/.style = {font=\tiny, color = black, auto=left, inner sep = 1pt},
]	
	\fill (0,0) ellipse [x radius=20pt, y radius=2pt] node[below=2pt]{\color{black}\tiny $B+\ell F$};
	\fill (0,2) ellipse [x radius=20pt, y radius=2pt] node[above=2pt]{\color{black}\tiny $B-\ell F$};
	\draw
	(-0.5,0) edge ["$F$"] (-0.5, 2);
\end{tikzpicture}
\caption{Decorated graphs for $(\Sigma\times S^2, \ow_{\lambda_F, \lambda_B})$ where $n=2\ell$ and $0\leq \ell < \lambda_B/\lambda_F$.}\labell{graph-ruled}
\end{figure}
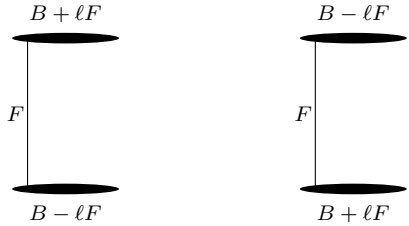

Holm--Kessler \cite{holmkessler} provide a characterization of Hamiltonian circle actions on blowups of $\Sigma \times S^2$ as follows.
\begin{Theorem} \label{charpositive} \cite[Theorem 2.13]{holmkessler}
For $k\geq 1$, 
let $\omega$ be a blowup form on $W_k$ whose cohomology class is encoded by a vector 
$(\lambda_F,\lambda_B;\delta_1,\ldots,\delta_k)$; if $k \geq 2$ assume that  $[\omega]$ is  
reduced. 
Then every Hamiltonian circle action on $(W_k,\omega)$ 
is obtained from a Hamiltonian circle action on 
$(\Sigma \times S^2, \ow_{\lambda_F,\lambda_B})$ by a 
sequence of $S^1$-equivariant symplectic blowups of sizes
$\delta_1, \delta_2, \dots, \delta_k$, in this order.
\end{Theorem}


\section{The example in the simply connected case} \label{ex-simply}


\begin{noTitle}\labell{example-construction}
{\bf We construct a $\Z_2$-action on $M_6$ with a blowup form $\omega_{1;\frac{1}{2},\frac{1}{4},\frac{1}{4},\frac{1}{4},\frac{3}{16},\frac{1}{8}}$} as follows:
Start from  a  $(\CP^2 \# \overline{\CP^2},\omega_{1;\frac{1}{2}})$ with the Hamiltonian circle action whose decorated graph is given in Figure \ref{graph-hirz-2} on the right with $\ell=1$.
The top fat vertex corresponds to an $S^1$-fixed symplectic sphere in $E_1$ of size $1/2$, and the bottom fat vertex corresponds to an $S^1$-fixed symplectic sphere in $L$ of size $1$.
Then perform three $S^1$-equivariant symplectic blowups of size $1/4$: one at the top fat vertex producing $E_2$, and two at the bottom fat vertex producing $E_3$ and $E_4$.
Next, perform an $S^1$-equivariant symplectic blowup of size $3/16$ at the interior fixed point created in the previous blowup. We obtain a Hamiltonian circle action on $M=M_5$ with a blowup form $\omega=\omega_{1;\frac{1}{2},\frac{1}{4},\frac{1}{4},\frac{1}{4},\frac{3}{16}}$
whose associated decorated graph, obtained as in Figures \ref{circlebl1}--\ref{circlebl3}, is shown on the left of Figure \ref{special}.
\begin{figure}[h]
\begin{tikzpicture}
[
	point/.style = {draw, circle,  fill = black, inner sep = 1pt},
	every edge quotes/.style = {font=\tiny, color = black, auto=left, inner sep = 1pt},
	loosestart/.style = {inner sep= 2pt, near start},
	looseend/.style = {inner sep= 2pt, near end},
]	 	
	\fill (0,0) ellipse [x radius=50pt, y radius=4pt] node[below=2pt]{\tiny $L-E_3-E_4$};
	\fill (0,2) ellipse [x radius=50pt, y radius=4pt] node[above=2pt]{\tiny $E_1-E_2$};	
	\draw node at (-1.5,1) [point]{};
	\draw node at (-1,1) [point]{};
	\draw node at (1.5,0.25) [point]{};
	\draw node at (1.5,1.75) [point]{};
	\draw
	(-1.5,0) edge ["$L-E_1-E_2$"] (-1.5,1)
	(-1.5,1) edge ["$E_2$"] (-1.5,2);
	\draw
	(-1,0) edge ["$E_3$"'] (-1,1)
	(-1,1) edge ["$L-E_1-E_3$"'] (-1,2);
	\draw
	(1.5,0) edge ["$E_4-E_5$"' looseend] (1.5,0.25)
	(1.5,0.25) edge ["$E_5$"', "$2$"] (1.5,1.75)
	(1.5,1.75) edge ["$L-E_1-E_4-E_5$"' loosestart] (1.5,2);
\end{tikzpicture}	
\begin{tikzpicture}
[
	point/.style = {draw, circle,  fill = black, inner sep = 1pt},
	every edge quotes/.style = {font=\tiny, color = black, auto=left, inner sep = 1pt},
	loosestart/.style = {inner sep= 2pt, near start},
	looseend/.style = {inner sep= 2pt, near end},
]	 	
	\fill (0,0) ellipse [x radius=50pt, y radius=4pt] node[below=2pt]{\tiny $L-E_3-E_4-E_6$};
	\fill (0,2) ellipse [x radius=50pt, y radius=4pt] node[above=2pt]{\tiny $E_1-E_2$};	
	\draw node at (-1.5,1) [point]{};
	\draw node at (-1.5,0.75) [point]{};
	\draw node at (-1,1) [point]{};
	\draw node at (1.5,0.25) [point]{};
	\draw node at (1.5,1.75) [point]{};
	\draw
	(-1.5,0) edge ["$E_6$"] (-1.5, 0.75)
	(-1.5,0.75) edge ["$L-E_1-E_2-E_6$"] (-1.5,1)
	(-1.5,1) edge ["$E_2$"] (-1.5,2);
	\draw
	(-1,0) edge ["$E_3$"'] (-1,1)
	(-1,1) edge ["$L-E_1-E_3$"'] (-1,2);
	\draw
	(1.5,0) edge ["$E_4-E_5$"' looseend] (1.5,0.25)
	(1.5,0.25) edge ["$E_5$"', "$2$"] (1.5,1.75)
	(1.5,1.75) edge ["$L-E_1-E_4-E_5$"' loosestart] (1.5,2);
\end{tikzpicture}	
\caption{Special decorated graphs.}\labell{special}
\end{figure}
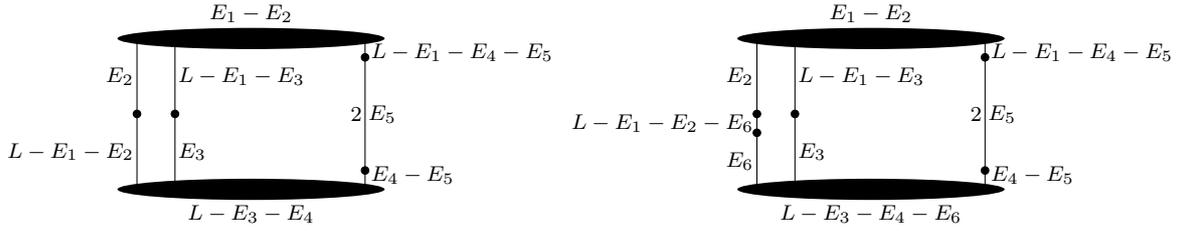
Note that the edge created in the last blowup is labelled $2$, i.e., it corresponds to a $\Z_{2}$-sphere in $E_5$.
By \cite[Theorem 7.1]{karshon}, $(M,\omega)$ admits an integrable complex structure such that the $S^1$-action is holomorphic and the symplectic form is K\"ahler.

Now perform a ${\Z_{2}}$-equivariant complex blowup at a point $p$ in the ${\Z}_{2}$-sphere in $E_5$ that is not an $S^1$-fixed point.
Let $$\pi \colon \tilde{M} \to M$$ be the equivariant complex blowup map. Let $\tilde{E_6}=\pi^{-1}(p)$ be the exceptional divisor and $\Xi_{6}$ be its Poincar\'e dual in $H^{2}(\tilde{M})$. Let
$$
\tilde{\Omega}:=\pi^{*}\Omega-\frac{1}{8} \Xi_{6},
$$
where $\Omega$ is the cohomology class of $\omega$.
We claim that $\tilde{\Omega}$ contains a K\"ahler form. By Nakai's criterion (see \cite{FM}, \cite[Lemma C.1]{karshon}, or \cite{MP:packing}, for example), the cohomology class $\tilde{\Omega}$ on the complex surface $\tilde{M}$ is represented by a K\"ahler form if and only if $\tilde{\Omega}^2 > 0$ and $\la \tilde{\Omega},[C]\ra > 0$ for every complex curve $C \subset \tilde{M}$.

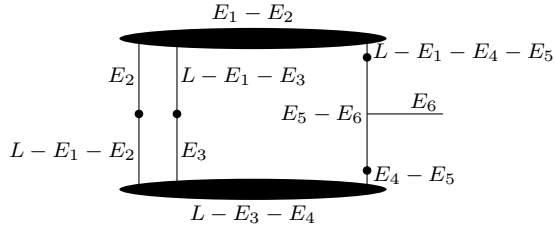
\begin{figure}[h]
\begin{tikzpicture}
[
	point/.style = {draw, circle,  fill = black, inner sep = 1pt},
	every edge quotes/.style = {font=\tiny, color = black, auto=left, inner sep = 1pt},
	loosestart/.style = {inner sep= 2pt, near start},
	looseend/.style = {inner sep= 2pt, near end},
]	 	
	\fill (0,0) ellipse [x radius=50pt, y radius=4pt] node[below=2pt]{\tiny $L-E_3-E_4$};
	\fill (0,2) ellipse [x radius=50pt, y radius=4pt] node[above=2pt]{\tiny $E_1-E_2$};	
	\draw node at (-1.5,1) [point]{};
	\draw node at (-1,1) [point]{};
	\draw node at (1.5,0.25) [point]{};
	\draw node at (1.5,1.75) [point]{};
	\draw
	(-1.5,0) edge ["$L-E_1-E_2$"] (-1.5,1)
	(-1.5,1) edge ["$E_2$"] (-1.5,2);
	\draw
	(-1,0) edge ["$E_3$"'] (-1,1)
	(-1,1) edge ["$L-E_1-E_3$"'] (-1,2);
	\draw
	(1.5,0) edge ["$E_4-E_5$"' looseend] (1.5,0.25)
	(1.5,0.25) edge ["$E_5-E_6$"] (1.5,1.75)
	(1.5,1.75) edge ["$L-E_1-E_4-E_5$"' loosestart] (1.5,2);
	\draw (1.5,1) edge ["$E_6$" near end] (2.5,1);
\end{tikzpicture}	
\caption{Special configuration.}\labell{countex}
\end{figure}

Note that $\tilde{M}$ is a weak del Pezzo surface of degree $9-6=3$ and type A4, see \cite{BW}.
Its ($-2$)- and ($-1$)-curves are in the classes
\begin{align}\labell{neg curves}
&E_4-E_5,\ E_5-E_6,\ L-E_1-E_4-E_5,\ E_1-E_2 \\
&E_6,\ L-E_3-E_4,\ E_2,\ L-E_1-E_3,\ L-E_1-E_2,\ E_3 \notag
\end{align}
as can be seen in Figure \ref{countex}. Note Figure \ref{countex} is not a decorated graph for a Hamiltonian $S^1$-space, but illustrates the $\Z_2$-invariant spheres on $\tM$.

The first 7 classes in \eqref{neg curves} form a basis of the Picard group and therefore the $\Z_2$-action induces the identity map on homology. All 10 classes in \eqref{neg curves}
minimally generate the effective cone (see, e.g., \cite[Proposition 5.2.1.10]{coxrings}).
Substituting the vector  $2\pi(1;\frac{1}{2},\frac{1}{4},\frac{1}{4},\frac{1}{4},\frac{3}{16},\frac{1}{8})$ as the pairing of $\tilde{\Omega}$ with $(L;E_1,\ldots,E_6)$, we verify that $\tilde{\Omega}^2> 0$ and that $\la \tilde{\Omega},[C] \ra > 0$ for each of these negative curves and hence for every irreducible curve. Therefore, by Nakai's criterion, $\tilde{\Omega}$ contains a K\"ahler form. By averaging with respect to the holomorphic $\Z_{2}$-action, (and since the action is identity on homology), we obtain an invariant K\"ahler form $\tilde{\omega}$ in $\tilde{\Omega}$. Note that the form $\tilde{\omega}$ is with a standard canonical class. Hence, by Corollary \ref{cord}, the symplectic form $\tilde{\omega}$ is a blowup form on  $\tilde{M}=M_{6}$.

Hence we obtain a homologically trivial symplectic $\Z_{2}$-action on $M_6$ with a blowup form $\omega_{1;\frac{1}{2},\frac{1}{4},\frac{1}{4},\frac{1}{4},\frac{3}{16},\frac{1}{8}}$. The adjunction formula implies that the $\Z_2$-fixed spheres in $E_1-E_2$, $E_5-E_6$, and $L-E_3-E_4$ are all embedded.

On $(M_5, \ow_{1;\frac{1}{2},\frac{1}{4},\frac{1}{4},\frac{1}{4},\frac{3}{16}})$ with the Hamiltonian circle action associated with the decorated graph on the left of Figure \ref{special}, we can instead perform an $S^1$-equivariant symplectic blowup to obtain the decorated graph shown on the right of Figure \ref{special}. This shows that $(M_6, \omega_{1;\frac{1}{2},\frac{1}{4},\frac{1}{4},\frac{1}{4},\frac{3}{16},\frac{1}{8}})$ admits a Hamiltonian $S^1$-action.

\end{noTitle}


\bigskip

\begin{noTitle}{\bf We claim that there is no Hamiltonian circle action that extends the constructed ${\Z}_{2}$-action.}
Assume that there is such a Hamiltonian circle action.

Note that the vector $(1;\frac{1}{2},\frac{1}{4},\frac{1}{4},\frac{1}{4},\frac{3}{16},\frac{1}{8})$ is reduced. By Theorem \ref{character}, this circle action on $(M_6,\omega_{1;\frac{1}{2},\frac{1}{4},\frac{1}{4},\frac{1}{4},\frac{3}{16},\frac{1}{8}})$ is obtained from a Hamiltonian circle action on $(\CP^2 \# \overline{\CP^2},\omega_{1;\frac{1}{2}})$ by performing five equivariant symplectic blowups of sizes $1/4$, $1/4$, $1/4$, $3/16$, $1/8$ in this order.
The possible decorated graphs for $(\CP^2 \# \overline{\CP^2},\omega_{1;\frac{1}{2}})$, up to the equivalence relation of flips and vertical translations, are the graphs in Figures \ref{graph-hirz-1} and \ref{graph-hirz-2} with $\ell=n=1$, since $1 \leq \ell < \lambda / (\lambda-\delta_1)= 1/(1-1/2)$. The possible decorated graphs for the $S^1$-equivariant symplectic blowups of the prescribed sizes are shown in Figures
\ref{blowup1-1a}--\ref{blowup4-4Rb}. At each stage, the color blue indicates insufficient size to carry out the next blowup in line (cf. Remark \ref{blowup-size}).

For the decorated graphs in Figure \ref{graph-hirz-1} and on the left of Figure \ref{graph-hirz-2},
the possible $S^1$-equivariant symplectic blowups that produce $E_2$ of size $1/4$ are shown in Figures \ref{blowup1-1a}, \ref{blowup1-1b}, \ref{blowup1-2}, and \ref{blowup1-3}. The possible $S^1$-equivariant symplectic blowups that produce $E_3$ of size $1/4$ subsequently are shown in Figures \ref{blowup2-1a}, \ref{blowup2-1b}, \ref{blowup2-2}, and \ref{blowup2-3}.
The $S^1$-equivariant symplectic blowups that produce $E_4$ of size $1/4$ next can no longer be carried out on these decorated graphs due to size conditions.

For the decorated graph on the right of Figure \ref{graph-hirz-2}, there are only two ways to perform three $S^1$-equivariant symplectic blowups that produce $E_2$, $E_3$, and $E_4$ of size $1/4$: their decorated graphs are shown in Figure \ref{blowup3-4}. The possible $S^1$-equivariant symplectic blowups that produce $E_5$ of size $3/16$ subsequently are shown in Figures \ref{blowup4-4La}--\ref{blowup4-4Rb}.

\begin{Remark}
	The following graphs can be interchanged by a change of compatible metric (or almost complex structure): the last two in Figure \ref{blowup1-3}, \ref{blowup2-3}, or \ref{blowup4-4Rb}; the two in Figure \ref{blowup4-4La} or \ref{blowup4-4Ra};  the three in Figure \ref{blowup4-4Lc}.
\end{Remark}

By Lemma \ref{clsta}, for every $S^1$-invariant compatible almost complex structure $J$, the embedded spheres of non-trivial stabilizers are $J$-holomorphic; hence the fat vertices and the edges of label greater than $1$ in the decorated graphs
correspond to $J$-holomorphic spheres.

After the last $S^1$-equivariant symplectic blowup that produces $E_6$ of size $1/8$, we will get a $J$-holomorphic sphere in one of the following classes:
\begin{itemize}
	\item $E_5$\,;
	\item $E_1-E_i-E_5$ for $i=2,\,3,\,4$\,;
	\item $E_1-E_i-E_6$ for $i=2,\,3,\,4,\,5$\,;
	\item $L-E_i-E_j-E_k$ with $i,\, j,\, k \in \{ 2,\, 3,\, 4,\, 5 \}$ and $i < j < k$\,;
	\item $L-E_i-E_j-E_k-E_\ell$ with $i,\, j,\, k,\, \ell \in \{ 2,\, 3,\, 4,\, 5,\, 6 \}$ and $i < j < k < \ell$.
\end{itemize}

By the construction of the ${\Z_2}$-action and the assumption that the circle action extends the $\Z_2$-action, there are $\Z_2$-fixed embedded spheres in $E_1-E_2$, $L-E_3-E_4$, and $E_5-E_6$.  By Lemma \ref{clsta} again, there are $J$-holomorphic spheres in these classes. Note that
\begin{itemize}
	\item $E_5 \cdot (E_5-E_6)=-1$\,;
	\item $(E_1-E_2-E_5) \cdot (E_1-E_2)=-2$, and $(E_1-E_{i}-E_5) \cdot (E_1-E_2)=-1$ for $i=3,\,4$\,;
	\item $(E_1-E_2-E_6) \cdot (E_1-E_2)=-2$,  and
$(E_1-E_{i}-E_6) \cdot (E_1-E_2)=-1$ for $i=3,\,4,\,5$\,;
\item $(L-E_2-E_j-E_k) \cdot (E_1 - E_2) =-1$, where $j,\, k \in \{ 3,\, 4,\, 5 \}$ and $ j < k$, and
$(L-E_3-E_4-E_5) \cdot (L-E_3-E_4) = -1$\,;
\item $(L-E_2-E_j-E_k-E_\ell) \cdot (E_1 - E_2) =-1$, where $j,\, k,\, \ell \in \{ 3,\, 4,\, 5,\, 6 \}$ and $j < k < \ell$, and
$(L-E_3-E_4-E_5-E_6) \cdot (L-E_3-E_4) = -1$.
\end{itemize}
In each of the above cases we have a contradiction to the positivity of intersections of $J$-holomorphic spheres \cite[Appendix E and Proposition 2.4.4]{nsmall}.

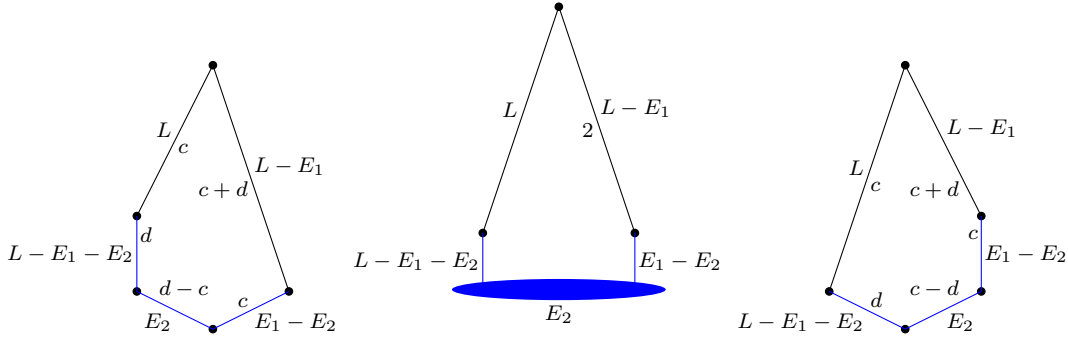
\begin{figure}[h]
\begin{tikzpicture}
	[
		>=stealth,
		point/.style = {draw, circle,  fill = black, inner sep = 1pt},
		every edge quotes/.style = {font=\tiny, color = black, auto=left, inner sep = 1pt},
	]
		\draw node at (0,1) [point]{};
		\draw node at (0,2) [point]{};
		\draw node at (1,4) [point]{};
		\draw node at (2,1) [point]{};
		\draw node at (1,0.5) [point]{};
		\draw
		(0, 1) edge ["$L-E_1-E_2$", "$d$"' near end, blue] (0, 2)
		(0, 2) edge ["$L$", "$c$"'] (1,4)
		(1, 4) edge ["$L-E_1$", "$c+d$"']	(2,1)
		(2,1) edge ["$E_1-E_2$", "$c$"', blue] (1,0.5)
		(1, 0.5) edge ["$E_2$", "$d-c$"' near end, blue] (0,1);
\end{tikzpicture}	
\begin{tikzpicture}
	[
		>=stealth,
		point/.style = {draw, circle,  fill = black, inner sep = 1pt},
		every edge quotes/.style = {font=\tiny, color = black, auto=left, inner sep = 1pt},
	]
	\fill[blue] (0,0.75) ellipse [x radius=40pt, y radius=4pt] node[below=2pt]{\color{black}\tiny $E_2$};
		\draw node at (-1,1.5) [point]{};
		\draw node at (0,4.5) [point]{};
		\draw node at (1,1.5) [point]{};
		\draw
		(-1, 0.75) edge ["$L-E_1-E_2$", blue] (-1, 1.5)
		(-1, 1.5) edge ["$L$"] (0,4.5)
		(0,4.5) edge ["$L-E_1$", "$2$"'] (1,1.5)
		(1,1.5) edge ["$E_1-E_2$", blue] (1,0.75);
\end{tikzpicture}	
\begin{tikzpicture}
[
	>=stealth,
	point/.style = {draw, circle,  fill = black, inner sep = 1pt},
	every edge quotes/.style = {font=\tiny, color = black, auto=left, inner sep=1pt},
]
	\draw node at (0,1) [point]{};
	\draw node at (1,4) [point]{};
	\draw node at (2,2) [point]{};
	\draw node at (2,1) [point]{};
	\draw node at (1,0.5) [point]{};
	\draw
	(0,1) edge ["$L$", "$c$"'] (1,4)
	(1,4) edge ["$L-E_1$", "$c+d$"' near end] (2,2)
	(2,2) edge ["$E_1-E_2$", "$c$"' near start, blue] (2,1)
	(2,1) edge ["$E_2$", "$c-d$"' near start, blue] (1,0.5)
	(1,0.5) edge ["$L-E_1-E_2$", "$d$"', blue] (0,1); 	
\end{tikzpicture}		
\caption{Blowup ($E_2$) at the minimum of the left graph of Figure \ref{graph-hirz-1} with $c<d$, $c=d=1$, $c>d$. }\labell{blowup1-1a}
\end{figure}


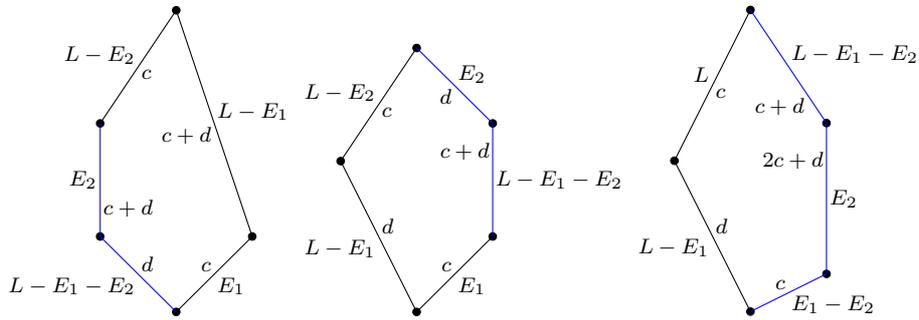
\begin{figure}[h]
\begin{tikzpicture}
[
	>=stealth,
	point/.style = {draw, circle,  fill = black, inner sep = 1pt},
	every edge quotes/.style = {font=\tiny, color = black, auto=left, inner sep = 1pt},
]
	\draw node at (0,1) [point]{};
	\draw node at (0,2.5) [point]{};
	\draw node at (1,4) [point]{};
	\draw node at (2,1) [point]{};
	\draw node at (1,0) [point]{};
	\draw
	(0, 1) edge ["$E_2$", "$c+d$"' near start, blue] (0, 2.5)
	(0, 2.5) edge ["$L-E_2$", "$c$"'] (1,4)
	(1,4) edge ["$L-E_1$", "$c+d$"'] (2,1)
	(2,1) edge ["$E_1$", "$c$"'] (1,0)
	(1,0) edge ["$L-E_1-E_2$", "$d$"', blue](0,1); 	
\end{tikzpicture}	
\begin{tikzpicture}
[
	>=stealth,
	point/.style = {draw, circle,  fill = black, inner sep = 1pt},
	every edge quotes/.style = {font=\tiny, color = black, auto=left, inner sep=1pt},
]
	\draw node at (0,2) [point]{};
	\draw node at (1,3.5) [point]{};
	\draw node at (2,2.5) [point]{};
	\draw node at (2,1) [point]{};
	\draw node at (1,0) [point]{};
	\draw
	(0, 2) edge ["$L-E_2$", "$c$"'] (1,3.5)
	(1,3.5) edge ["$E_2$", "$d$"', blue] (2,2.5)
	(2,2.5) edge ["$L-E_1-E_2$", "$c+d$"' near start, blue] (2,1)
	(2,1) edge ["$E_1$", "$c$"'] (1,0)
	(1,0) edge ["$L-E_1$", "$d$"'] (0,2); 	
\end{tikzpicture}	
\begin{tikzpicture}
[
	>=stealth,
	point/.style = {draw, circle,  fill = black, inner sep = 1pt},
	every edge quotes/.style = {font=\tiny, color = black, auto=left, inner sep = 1pt},
]
	\draw node at (0,2) [point]{};
	\draw node at (1,4) [point]{};
	\draw node at (2,2.5) [point]{};
	\draw node at (2,0.5) [point]{};
	\draw node at (1,0) [point]{};
	\draw
	(0,2) edge ["$L$", "$c$"'] (1,4)
	(1,4) edge ["$L-E_1-E_2$", "$c+d$"' near end, blue] (2,2.5)
	(2,2.5) edge ["$E_2$", "$2c+d$"' near start, blue] (2,0.5)
	(2,0.5) edge ["$E_1-E_2$", "$c$"' , blue] (1,0)
	(1,0) edge ["$L-E_1$", "$d$"' ] (0,2); 	
\end{tikzpicture}	
\caption{Blowup ($E_2$) at a non-minimal fixed point of the left graph of Figure \ref{graph-hirz-1}.}\labell{blowup1-1b}
\end{figure}


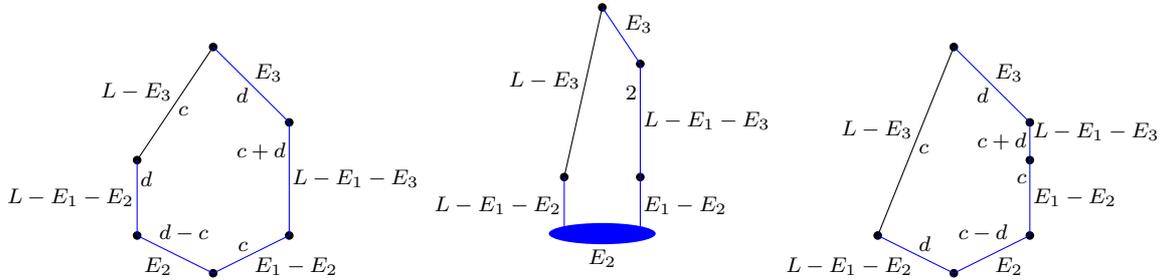
\begin{figure}[h]
\begin{tikzpicture}
	[
		>=stealth,
		point/.style = {draw, circle,  fill = black, inner sep = 1pt},
		every edge quotes/.style = {font=\tiny, color = black, auto=left, inner sep = 1pt},
	]
		\draw node at (0,1) [point]{};
		\draw node at (0,2) [point]{};
		\draw node at (1,3.5) [point]{};
		\draw node at (2,2.5) [point]{};
		\draw node at (2,1) [point]{};
		\draw node at (1,0.5) [point]{};
		\draw
		(0, 1) edge ["$L-E_1-E_2$", "$d$"' near end, blue] (0, 2)
		(0,2) edge ["$L-E_3$", "$c$"'] (1,3.5)
		(1,3.5) edge ["$E_3$", "$d$"', blue] (2,2.5)
		(2,2.5) edge ["$L-E_1-E_3$", "$c+d$"' near start, blue] (2,1)	
		(2,1) edge ["$E_1-E_2$", "$c$"', blue] (1,0.5)
		(1,0.5) edge ["$E_2$", "$d-c$"' near end, blue] (0,1); 	
\end{tikzpicture}	
\begin{tikzpicture}
	[
		>=stealth,
		point/.style = {draw, circle,  fill = black, inner sep = 1pt},
		every edge quotes/.style = {font=\tiny, color = black, auto=left, inner sep = 1pt},
	]
	\fill[blue] (0,0.75) ellipse [x radius=20pt, y radius=4pt] node[below=2pt]{\color{black}\tiny $E_2$};
		\draw node at (-0.5,1.5) [point]{};
		\draw node at (0,3.75) [point]{};
		\draw node at (0.5,3) [point]{};
		\draw node at (0.5,1.5) [point]{};
		\draw
		(-0.5, 0.75) edge ["$L-E_1-E_2$", blue] (-0.5, 1.5)
		(-0.5,1.5) edge ["$L-E_3$"] (0,3.75)
		(0,3.75) edge ["$E_3$", blue] (0.5,3)
		(0.5,3) edge ["$L-E_1-E_3$", "$2$"' near start, blue] (0.5,1.5)
		(0.5, 1.5) edge ["$E_1-E_2$", blue] (0.5, 0.75);
\end{tikzpicture}	
\begin{tikzpicture}
[
	>=stealth,
	point/.style = {draw, circle,  fill = black, inner sep = 1pt},
	every edge quotes/.style = {font=\tiny, color = black, auto=left, inner sep = 1pt},
]
	\draw node at (0,1) [point]{};
	\draw node at (1,3.5) [point]{};
	\draw node at (2,2.5) [point]{};
	\draw node at (2,2) [point]{};
	\draw node at (2,1) [point]{};
	\draw node at (1,0.5) [point]{};
	\draw
	(0, 1) edge ["$L-E_3$", "$c$"' ] (1,3.5)
	(1,3.5) edge ["$E_3$", "$d$"' , blue] (2,2.5)
	(2,2.5) edge ["$L-E_1-E_3$" near start, "$c+d$"', blue] (2,2)
	(2,2) edge ["$E_1-E_2$", "$c$"' near start, blue] (2,1)
	(2,1) edge ["$E_2$", "$c-d$"' near start, blue] (1,0.5)
	(1,0.5) edge ["$L-E_1-E_2$", "$d$"' , blue] (0,1); 	
\end{tikzpicture}		
\caption{Blowup ($E_3$) at the maximum from Figure \ref{blowup1-1a}. Blowup ($E_4$) can not be performed afterward due to size consideration.}\labell{blowup2-1a}
\end{figure}


\begin{figure}[h]
\begin{tikzpicture}
	[
		>=stealth,
		point/.style = {draw, circle,  fill = black, inner sep = 1pt},
		every edge quotes/.style = {font=\tiny, color = black, auto=left, inner sep = 1pt},
	]
		\draw node at (0,1) [point]{};
		\draw node at (0,2.5) [point]{};
		\draw node at (1,4) [point]{};
		\draw node at (2,2.5) [point]{};
		\draw node at (2,0.5) [point]{};
		\draw node at (1,0) [point]{};
		\draw
		(0, 1) edge ["$E_2$", "$c+d$"' , blue] (0, 2.5)
		(0,2.5) edge ["$L-E_2$", "$c$"' ] (1,4)
		(1,4) edge ["$L-E_1-E_3$", "$c+d$"' near end, blue] (2,2.5)
		(2,2.5) edge ["$E_3$", "$2c+d$"' , blue] (2,0.5)	
		(2, 0.5) edge ["$E_1-E_3$", "$c$"' , blue] (1,0)
		(1,0) edge ["$L-E_1-E_2$", "$d$"' , blue] (0,1); 	
\end{tikzpicture}
\begin{tikzpicture}
	[
		>=stealth,
		point/.style = {draw, circle,  fill = black, inner sep = 1pt},
		every edge quotes/.style = {font=\tiny, color = black, auto=left, inner sep = 1pt},
	]
		\draw node at (0,1) [point]{};
		\draw node at (0,2.5) [point]{};
		\draw node at (1,3.5) [point]{};
		\draw node at (2,2.5) [point]{};
		\draw node at (2,1) [point]{};
		\draw node at (1,0) [point]{};
		\draw
		(0, 1) edge ["$E_2$", "$c+d$"', blue] (0, 2.5)
		(0,2.5) edge ["$L-E_2-E_3$", "$c$"'] (1,3.5)
		(1,3.5) edge ["$E_3$", "$d$"', blue] (2,2.5)
		(2,2.5) edge ["$L-E_1-E_3$", "$c+d$"', blue] (2,1)	
		(2,1) edge ["$E_1$", "$c$"'] (1,0)
		(1,0) edge ["$L-E_1-E_2$", "$d$"', blue] (0,1); 	
\end{tikzpicture}
\caption{Blowup ($E_3$) from Figure \ref{blowup1-1b} at a non-minimal fixed point, or with $E_2$, $E_3$ switched. Blowup ($E_3$) at a minimum from Figure \ref{blowup1-1b} goes back to blowup ($E_3$) from Figure \ref{blowup1-1a} with $E_2$, $E_3$ switched. Again the blowup ($E_4$) can not be performed afterward due to size consideration.}\labell{blowup2-1b}
\end{figure}
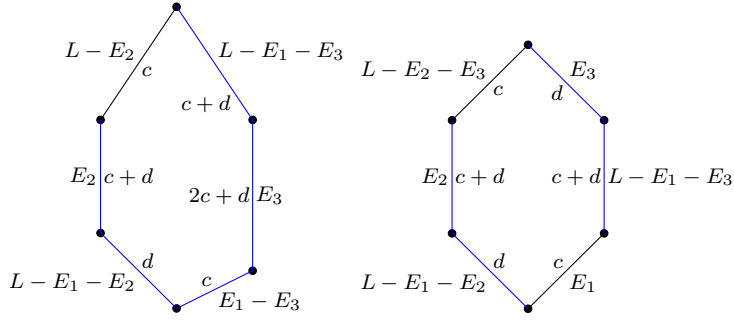

\vfill
\pagebreak

\begin{figure}[h]
\begin{tikzpicture}
	[
		point/.style = {draw, circle,  fill = black, inner sep = 1pt},
		every edge quotes/.style = {font=\tiny, color = black, auto=left, inner sep = 1pt},
	]
		\draw node at (0,0.5) [point]{};
		\draw node at (0,2) [point]{};
		\draw node at (0,3) [point]{};
		\draw node at (2,4) [point]{};
		\draw node at (2,0) [point]{};
		\draw
		(0, 0.5) edge ["$E_2$", "$c+d$"', blue] (0, 2)
		(0, 2) edge ["$E_1-E_2$", "$c$"', blue] (0,3)
		(0,3) edge ["$L-E_1$", "$c-d$"'] (2,4)
		(2,4) edge ["$L$", "$c$"'] (2,0)
		(2,0) edge ["$L-E_1-E_2$", "$d$"', blue] (0,0.5); 	
\end{tikzpicture}		
\begin{tikzpicture}
	[
		point/.style = {draw, circle,  fill = black, inner sep = 1pt},
		every edge quotes/.style = {font=\tiny, color = black, auto=left, inner sep = 1pt},
	]
		\draw node at (0,1) [point]{};
		\draw node at (0,2) [point]{};
		\draw node at (0,3.5) [point]{};
		\draw node at (2,4) [point]{};
		\draw node at (2,0) [point]{};
		\draw
		(0, 1) edge ["$E_1-E_2$", "$c$"', blue] (0, 2)
		(0,2) edge ["$E_2$", "$2c-d$"', blue] (0,3.5)
		(0,3.5) edge ["$L-E_1-E_2$", "$c-d$"', blue] (2,4)
		(2,4) edge ["$L$", "$c$"'] (2,0)
		(2,0) edge ["$L-E_1$", "$d$"'] (0,1); 	
\end{tikzpicture}	
\begin{tikzpicture}
	[
		point/.style = {draw, circle,  fill = black, inner sep = 1pt},
		every edge quotes/.style = {font=\tiny, color = black, auto=left, inner sep = 1pt},
	]
		\draw node at (0,1) [point]{};
		\draw node at (0,3) [point]{};
		\draw node at (1,3.5) [point]{};
		\draw node at (2,3) [point]{};
		\draw node at (1,0) [point]{};
		\draw
		(0, 1) edge ["$E_1$", "$c$"'] (0, 3)
		(0, 3) edge ["$L-E_1-E_2$", "$c-d$"' near start, blue] (1,3.5)
		(1, 3.5) edge ["$E_2$", "$d$"', blue] (2,3)
		(2,3) edge ["$L-E_2$", "$c$"'] (1,0)
		(1,0) edge ["$L-E_1$", "$d$"'] (0,1); 	
\end{tikzpicture}	
\begin{tikzpicture}
	[
		point/.style = {draw, circle,  fill = black, inner sep = 1pt},
		every edge quotes/.style = {font=\tiny, color = black, auto=left, inner sep = 1pt},
	]
		\draw node at (0,1) [point]{};
		\draw node at (0,3) [point]{};
		\draw node at (1,4) [point]{};
		\draw node at (2,1) [point]{};
		\draw node at (1,0.5) [point]{};
		\draw
		(0, 1) edge ["$E_1$", "$c$"'] (0, 3)
		(0, 3) edge ["$L-E_1$", "$c-d$"' near start] (1,4)
		(1, 4) edge ["$L-E_2$", "$c$"'] (2,1)
		(2, 1) edge ["$E_2$", "$c-d$"' near start, blue] (1,0.5)
		(1, 0.5) edge ["$L-E_1-E_2$", "$d$"', blue] (0,1); 	
\end{tikzpicture}	
\caption{Blowup ($E_2$) of the right graph of Figure \ref{graph-hirz-1}. Note $c > c-d \geq 1$ and so $2c-d >1$.}\labell{blowup1-2}
\end{figure}
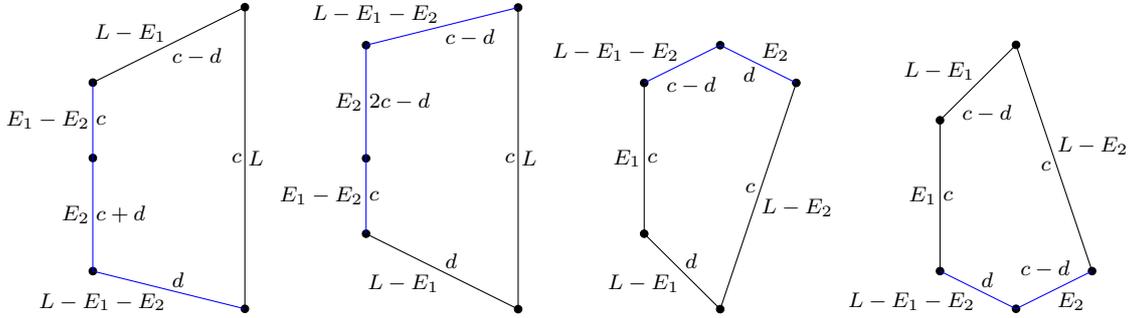


\begin{figure}[h]
\begin{tikzpicture}
	[
		point/.style = {draw, circle,  fill = black, inner sep = 1pt},
		every edge quotes/.style = {font=\tiny, color = black, auto=left, inner sep = 1pt},
	]
		\draw node at (0,0.5) [point]{};
		\draw node at (0,2) [point]{};
		\draw node at (0,3) [point]{};
		\draw node at (1,3.5) [point]{};
		\draw node at (2,3) [point]{};
		\draw node at (1,0) [point]{};
		\draw
		(0, 0.5) edge ["$E_2$", "$c+d$"', blue] (0, 2)
		(0, 2) edge ["$E_1-E_2$", "$c$"', blue] (0,3)
		(0, 3) edge ["$L-E_1-E_3$", "$c-d$"' near start, blue] (1,3.5)
		(1, 3.5) edge ["$E_3$", "$d$"', blue] (2,3)
		(2, 3) edge ["$L-E_3$", "$c$"'] (1,0)
		(1, 0) edge ["$L-E_1-E_2$", "$d$"', blue] (0,0.5); 	
\end{tikzpicture}		
\begin{tikzpicture}
	[
		point/.style = {draw, circle,  fill = black, inner sep = 1pt},
		every edge quotes/.style = {font=\tiny, color = black, auto=left, inner sep = 1pt},
	]
		\draw node at (0,1) [point]{};
		\draw node at (0,2) [point]{};
		\draw node at (0,3.5) [point]{};
		\draw node at (1,4) [point]{};
		\draw node at (2,1) [point]{};
		\draw node at (1,0.5) [point]{};
		\draw
		(0, 1) edge ["$E_1-E_2$", "$c$"', blue] (0, 2)
		(0, 2) edge ["$E_2$", "$2c-d$"', blue] (0,3.5)
		(0, 3.5) edge ["$L-E_1-E_2$", "$c-d$"' near start, blue] (1,4)
		(1, 4) edge ["$L-E_3$", "$c$"'] (2,1)
		(2, 1) edge ["$E_3$", "$c-d$"' near start, blue] (1,0.5)
		(1, 0.5) edge ["$L-E_1-E_3$", "$d$"', blue] (0,1); 	
\end{tikzpicture}	
\begin{tikzpicture}
	[
		point/.style = {draw, circle,  fill = black, inner sep = 1pt},
		every edge quotes/.style = {font=\tiny, color = black, auto=left, inner sep = 1pt},
	]	
		\draw node at (0,1) [point]{};
		\draw node at (0,3) [point]{};
		\draw node at (1,3.5) [point]{};
		\draw node at (2,3) [point]{};
		\draw node at (2,1) [point]{};
		\draw node at (1,0.5) [point]{};
		\draw
		(0, 1) edge ["$E_1$", "$c$"'] (0, 3)
		(0, 3) edge ["$L-E_1-E_2$", "$c-d$"' near start, blue] (1,3.5)
		(1, 3.5) edge ["$E_2$", "$d$"', blue] (2,3)
		(2, 3) edge ["$L-E_2-E_3$", "$c$"'] (2,1)
		(2, 1) edge ["$E_3$", "$c-d$"' near start, blue] (1,0.5)
		(1, 0.5) edge ["$L-E_1-E_3$", "$d$"', blue] (0,1); 	
\end{tikzpicture}	

\caption{Blowup ($E_3$) from Figure \ref{blowup1-2} or with $E_2$, $E_3$ switched. Due to size conditions, the blowup ($E_4$) can not be performed afterward.}\labell{blowup2-2}
\end{figure}
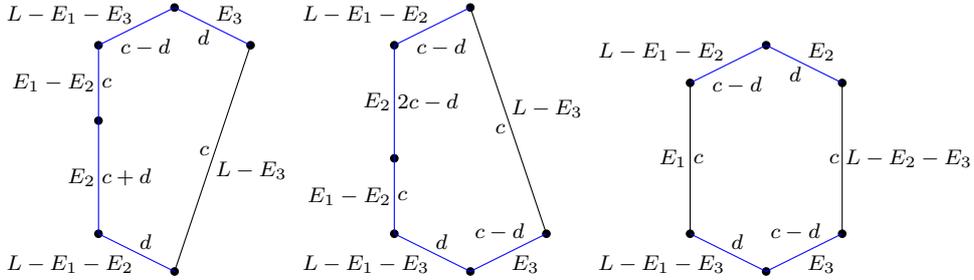


\begin{figure}[h]
\begin{tikzpicture}
[
	>=stealth,
	point/.style = {draw, circle,  fill = black, inner sep = 1pt},
	every edge quotes/.style = {font=\tiny, color = black, auto=left, inner sep = 1pt},
]	
	\draw node at (-0.5,0.5) [point]{};
	\draw node at (-0.5,1.5) [point]{};
	\draw node at (0.5,2) [point]{};
	\fill (0,0) ellipse [x radius=20pt, y radius=4pt] node[below=2pt]{\tiny $L-E_1$};
	\draw
	(-0.5,0) edge ["$E_1-E_2$", blue] (-0.5, 0.5)
	(-0.5, 0.5) edge ["$E_2$", "$2$"', blue] (-0.5,1.5)
	(-0.5, 1.5) edge ["$L-E_1-E_2$", blue] (0.5,2)
	(0.5, 2) edge ["$L$"] (0.5,0);
\end{tikzpicture}
\begin{tikzpicture}
[
	>=stealth,
	point/.style = {draw, circle,  fill = black, inner sep = 1pt},
	every edge quotes/.style = {font=\tiny, color = black, auto=left, inner sep = 1pt},
]	
	\draw node at (-0.5,1) [point]{};
	\fill (0,0) ellipse [x radius=20pt, y radius=4pt] node[below=2pt]{\tiny $L-E_1$};
	\fill[blue] (0,1.5) ellipse [x radius=20pt, y radius=4pt] node[above=2pt]{\color{black}\tiny $E_2$};
	\draw
	(-0.5,0) edge["$E_1$"] (-0.5, 1)
	(-0.5, 1) edge ["$L-E_1-E_2$", blue] (-0.5,1.5);
	\draw
	(0.5,1.5) edge ["$L-E_2$"] (0.5, 0);
\end{tikzpicture}
\begin{tikzpicture}
[
	>=stealth,
	point/.style = {draw, circle,  fill = black, inner sep = 1pt},
	every edge quotes/.style = {font=\tiny, color = black, auto=left, inner sep = 1pt},
]	
	\draw node at (-0.5,0.5) [point]{};
	\draw node at (-0.5,1) [point]{};
	\draw node at (0.5,2) [point]{};
	\fill[blue] (0,0) ellipse [x radius=20pt, y radius=4pt] node[below=2pt]{\color{black}\tiny $L-E_1-E_2$};
	\draw
	(-0.5,0) edge ["$E_2$", blue] (-0.5, 0.5)
	(-0.5, 0.5) edge ["$E_1-E_2$", blue] (-0.5,1)
	(-0.5, 1) edge ["$L-E_1$"] (0.5,2)
	(0.5, 2) edge ["$L$"] (0.5,0);
\end{tikzpicture}
\begin{tikzpicture}
[
	>=stealth,
	point/.style = {draw, circle,  fill = black, inner sep = 1pt},
	every edge quotes/.style = {font=\tiny, color = black, auto=left, inner sep = 1pt},
]	
	\draw node at (-0.5,1) [point]{};
	\draw node at (0.5,0.5) [point]{};
	\draw node at (0.5,2) [point]{};
	\fill[blue] (0,0) ellipse [x radius=20pt, y radius=4pt] node[below=2pt]{\color{black}\tiny $L-E_1-E_2$};
	\draw
	(-0.5,0) edge ["$E_1$"] (-0.5, 1)
	(-0.5, 1) edge ["$L-E_1$"] (0.5,2)
	(0.5, 2) edge ["$L-E_2$"] (0.5,0.5)
	(0.5, 0.5) edge ["$E_2$", blue] (0.5,0);
\end{tikzpicture}
\caption{Blowup ($E_2$) of the left graph of Figure \ref{graph-hirz-2}.
}\labell{blowup1-3}
\end{figure}
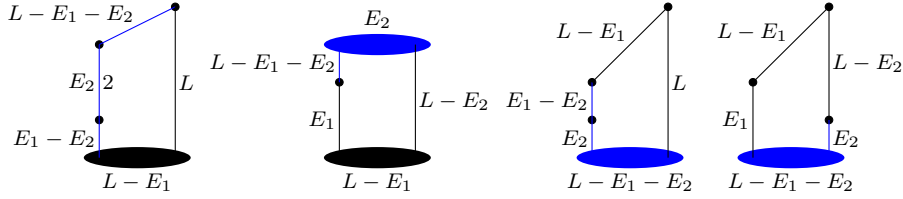

\vfill
\pagebreak


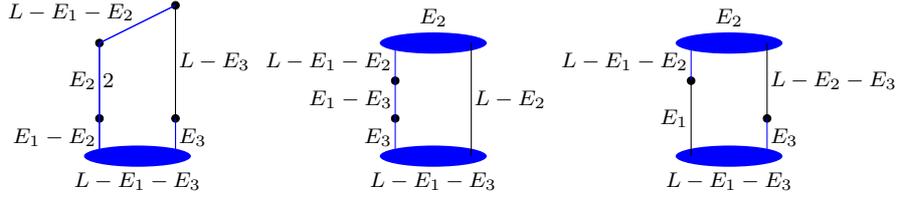
\begin{figure}[h]
\begin{tikzpicture}
[
	>=stealth,
	point/.style = {draw, circle,  fill = black, inner sep = 1pt},
	every edge quotes/.style = {font=\tiny, color = black, auto=left, inner sep = 1pt},
]	
	\draw node at (-0.5,0.5) [point]{};
	\draw node at (-0.5,1.5) [point]{};
	\draw node at (0.5,2) [point]{};
	\draw node at (0.5,0.5) [point]{};
	\fill[blue] (0,0) ellipse [x radius=20pt, y radius=4pt] node[below=2pt]{\color{black}\tiny $L-E_1-E_3$};
	\draw
	(-0.5,0) edge ["$E_1-E_2$", blue] (-0.5, 0.5)
	(-0.5,0.5) edge ["$E_2$", "$2$"', blue] (-0.5,1.5)
	(-0.5, 1.5) edge ["$L-E_1-E_2$", blue] (0.5,2)
	(0.5, 2) edge ["$L-E_3$"] (0.5,0.5)
	(0.5, 0.5) edge ["$E_3$", blue] (0.5,0);
	\draw[blue] (-0.5,0) -- (-0.5,0.5) -- (-0.5,1.5) -- (0.5,2) (0.5,0.5)--(0.5,0);
\end{tikzpicture}
\begin{tikzpicture}
[
	>=stealth,
	point/.style = {draw, circle,  fill = black, inner sep = 1pt},
	every edge quotes/.style = {font=\tiny, color = black, auto=left, inner sep = 1pt},
]	
	\draw node at (-0.5,0.5) [point]{};
	\draw node at (-0.5,1) [point]{};
	\fill[blue] (0,0) ellipse [x radius=20pt, y radius=4pt] node[below=2pt]{\color{black}\tiny $L-E_1-E_3$};
	\fill[blue] (0,1.5) ellipse [x radius=20pt, y radius=4pt] node[above=2pt]{\color{black}\tiny $E_2$};
	\draw
	(-0.5,0) edge ["$E_3$", blue] (-0.5, 0.5)
	(-0.5, 0.5) edge ["$E_1-E_3$", blue] (-0.5,1)
	(-0.5, 1) edge ["$L-E_1-E_2$", blue] (-0.5,1.5);
	\draw
	(0.5,1.5) edge ["$L-E_2$"] (0.5,0);
\end{tikzpicture}
\begin{tikzpicture}
[
	>=stealth,
	point/.style = {draw, circle,  fill = black, inner sep = 1pt},
	every edge quotes/.style = {font=\tiny, color = black, auto=left, inner sep = 1pt},
]	
	\draw node at (-0.5,1) [point]{};
	\draw node at (0.5,0.5) [point]{};
	\fill[blue] (0,0) ellipse [x radius=20pt, y radius=4pt] node[below=2pt]{\color{black}\tiny $L-E_1-E_3$};
	\fill[blue] (0,1.5) ellipse [x radius=20pt, y radius=4pt] node[above=2pt]{\color{black}\tiny $E_2$};
	\draw
	(-0.5,0) edge ["$E_1$"] (-0.5, 1)
	(-0.5, 1) edge ["$L-E_1-E_2$", blue] (-0.5,1.5);
	\draw
	(0.5,0) edge ["$E_3$"', blue] (0.5,0.5)
	(0.5, 0.5) edge ["$L-E_2-E_3$"'] (0.5,1.5);
\end{tikzpicture}
\caption{Blowup ($E_3$) from Figure \ref{blowup1-3} or with $E_2$, $E_3$ switched. The blowup ($E_4$) can not be performed afterward due to size condition.}\labell{blowup2-3}
\end{figure}


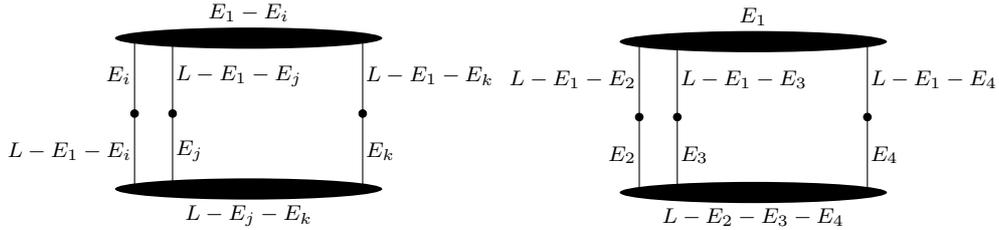
\begin{figure}[h]
\begin{tikzpicture}
[
	>=stealth,
	point/.style = {draw, circle,  fill = black, inner sep = 1pt},
	every edge quotes/.style = {font=\tiny, color = black, auto=left, inner sep = 1pt},
]	
	\draw node at (-1.5,1) [point]{};
	\draw node at (-1,1) [point]{};
	\draw node at (1.5,1) [point]{};
	\fill (0,0) ellipse [x radius=50pt, y radius=4pt] node[below=2pt]{\tiny $L-E_j-E_k$};
	\fill (0,2) ellipse [x radius=50pt, y radius=4pt] node[above=2pt]{\tiny $E_1-E_i$};
	\draw
	(-1.5,0) edge ["$L-E_1-E_i$"] (-1.5, 1)
	(-1.5, 1) edge ["$E_i$"] (-1.5,2);
	\draw
	(-1,0) edge ["$E_j$"'] (-1, 1)
	(-1, 1) edge ["$L-E_1-E_j$"'] (-1,2);
	\draw
	(1.5,0) edge ["$E_k$"'] (1.5, 1)
	(1.5,1) edge ["$L-E_1-E_k$"'] (1.5,2);
\end{tikzpicture}
\begin{tikzpicture}
[
	>=stealth,
	point/.style = {draw, circle,  fill = black, inner sep = 1pt},
	every edge quotes/.style = {font=\tiny, color = black, auto=left, inner sep = 1pt},
]	
	\draw node at (-1.5,1) [point]{};
	\draw node at (-1,1) [point]{};
	\draw node at (1.5,1) [point]{};
	\fill (0,0) ellipse [x radius=50pt, y radius=4pt] node[below=2pt]{\tiny $L-E_2-E_3-E_4$};
	\fill (0,2) ellipse [x radius=50pt, y radius=4pt] node[above=2pt]{\tiny $E_1$};
	\draw
	(-1.5,0) edge ["$E_2$"] (-1.5, 1)
	(-1.5, 1) edge ["$L-E_1-E_2$"] (-1.5,2);
	\draw
	(-1,0) edge ["$E_3$"'] (-1, 1)
	(-1, 1) edge ["$L-E_1-E_3$"'] (-1,2);
	\draw
	(1.5,0) edge ["$E_4$"'] (1.5, 1)
	(1.5, 1) edge ["$L-E_1-E_4$"'] (1.5,2);
\end{tikzpicture}
\caption{Three blowups of the right graph of Figure \ref{graph-hirz-2} with $\{i, j, k\} = \{2, 3, 4\}$. Due to size conditions, at most one blowup can be performed at the maximal surface. None of these blowups can be performed at an interior fixed point created by a previous blowup.}\labell{blowup3-4}
\end{figure}


\begin{figure}[h]
\begin{tikzpicture}
[
	>=stealth,
	point/.style = {draw, circle,  fill = black, inner sep = 1pt},
]	
	\fill[blue] (0,2) ellipse [x radius=40pt, y radius=4pt] node[above=2pt]{\color{black}\tiny $E_1-E_i-E_5$};
	\draw[blue] (-1, 1.7) -- (-1,2);
	\draw (0, 1.7) -- (0,2);
	\draw (1, 1.7) -- (1,2);
\end{tikzpicture}	
\hspace{1cm}
\begin{tikzpicture}
[
	>=stealth,
	point/.style = {draw, circle,  fill = black, inner sep = 1pt},
]	
	\fill[blue] (0,2) ellipse [x radius=40pt, y radius=4pt] node[above=2pt]{\color{black}\tiny $E_1-E_i-E_5$};
	\draw (-1, 1.7) -- (-1,2);
	\draw (-0.5, 1.7) -- (-0.5,2);
	\draw (0.5, 1.7) -- (0.5,2);	
	\draw (1, 1.7) -- (1,2);
\end{tikzpicture}	
\caption{	
If the blowup ($E_5$) is performed on the maximal surface of the left graph of Figure \ref{blowup3-4}, then $E_1-E_i-E_5$, $i \in \{2,3,4\}$, remains after the next blowup ($E_6$) because of size consideration.
}\labell{blowup4-4La}
\end{figure}
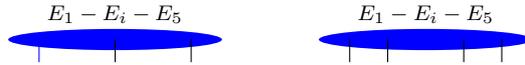

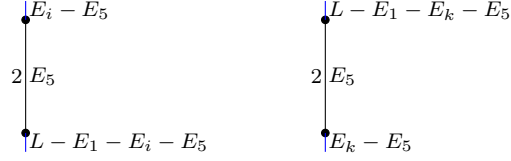
\begin{figure}[h]
\begin{tikzpicture}
[
	>=stealth,
	point/.style = {draw, circle,  fill = black, inner sep = 1pt},
	every edge quotes/.style = {font=\tiny, color = black, auto=left, inner sep = 1pt},
]	
	\draw node at (0,0.25) [point]{};
	\draw node at (0,1.75) [point]{};
	\draw
	(0,0) edge ["$L-E_1-E_i - E_5$"', blue] (0,0.25)
	(0, 0.25) edge ["$E_5$"', "$2$"] (0,1.75)
	(0, 1.75) edge ["$E_i-E_5$"', blue](0,2);
\end{tikzpicture}		
\hspace{1cm}
\begin{tikzpicture}
[
	>=stealth,
	point/.style = {draw, circle,  fill = black, inner sep = 1pt},
	every edge quotes/.style = {font=\tiny, color = black, auto=left, inner sep = 1pt},
]	
	\draw node at (0,0.25) [point]{};
	\draw node at (0,1.75) [point]{};
	\draw (0,0) edge ["$E_k - E_5$"', blue] (0,0.25)
	(0, 0.25) edge ["$E_5$"', "$2$"] (0,1.75)
	(0, 1.75) edge ["$L-E_1-E_k-E_5$"', blue] (0,2);
\end{tikzpicture}	
\caption{	
If the blowup ($E_5$) is performed at an interior fixed point of the left graph of Figure \ref{blowup3-4}, then $E_5$ remains after the next blowup ($E_6$) due to size consideration. Here $i, k \in \{2,3,4\}$.}\labell{blowup4-4Lb}
\end{figure}

\vfill
\pagebreak

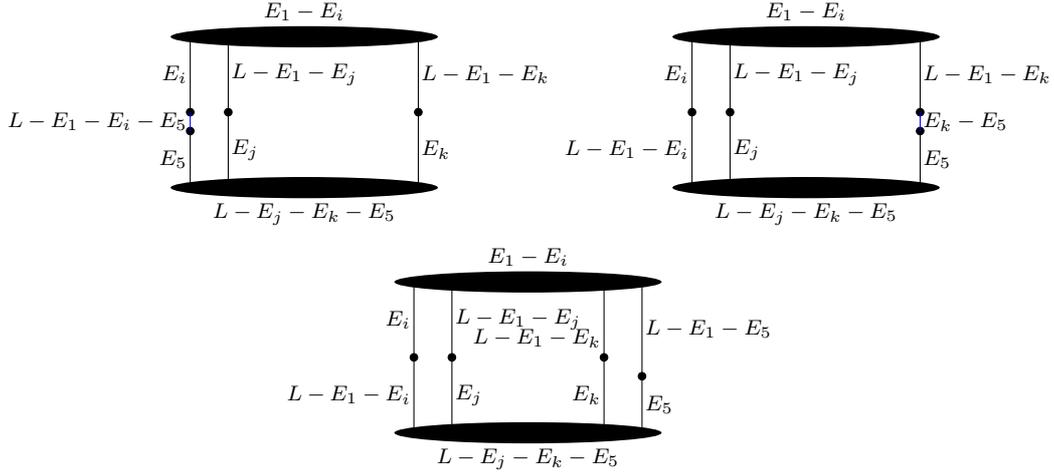
\begin{figure}[h]
\begin{tikzpicture}
[
	>=stealth,
	point/.style = {draw, circle,  fill = black, inner sep = 1pt},
	every edge quotes/.style = {font=\tiny, color = black, auto=left, inner sep = 1pt},
]	
	\draw node at (-1.5,0.75) [point]{};
	\draw node at (-1.5,1) [point]{};
	\draw node at (-1,1) [point]{};
	\draw node at (1.5,1) [point]{};
	\fill (0,0) ellipse [x radius=50pt, y radius=4pt] node[below=2pt]{\tiny $L-E_j-E_k-E_5$};
	\fill (0,2) ellipse [x radius=50pt, y radius=4pt] node[above=2pt]{\tiny $E_1-E_i$};
	\draw
	(-1.5,0) edge ["$E_5$"] (-1.5, 0.75)
	(-1.5, 0.75) edge ["$L-E_1-E_i-E_5$", blue] (-1.5, 1)
	(-1.5, 1) edge ["$E_i$"] (-1.5,2);
	\draw
	(-1,0) edge ["$E_j$"'] (-1, 1)
	(-1,1) edge ["$L-E_1-E_j$"'] (-1,2);
	\draw
	(1.5,0) edge ["$E_k$"'] (1.5, 1)
	(1.5, 1) edge ["$L-E_1-E_k$"'] (1.5,2);
\end{tikzpicture}
\begin{tikzpicture}
[
	>=stealth,
	point/.style = {draw, circle,  fill = black, inner sep = 1pt},
	every edge quotes/.style = {font=\tiny, color = black, auto=left, inner sep = 1pt},
]	
	\draw node at (-1.5,1) [point]{};
	\draw node at (-1,1) [point]{};
	\draw node at (1.5,0.75) [point]{};
	\draw node at (1.5,1) [point]{};
	\fill (0,0) ellipse [x radius=50pt, y radius=4pt] node[below=2pt]{\tiny $L-E_j-E_k-E_5$};
	\fill (0,2) ellipse [x radius=50pt, y radius=4pt] node[above=2pt]{\tiny $E_1-E_i$};
	\draw
	(-1.5,0) edge ["$L-E_1-E_i$"] (-1.5, 1)
	(-1.5, 1) edge ["$E_i$"] (-1.5,2);
	\draw
	(-1,0) edge ["$E_j$"'] (-1, 1)
	(-1, 1) edge ["$L-E_1-E_j$"'] (-1,2);
	\draw
	(1.5,0) edge ["$E_5$"'] (1.5, 0.75)
	(1.5, 0.75) edge ["$E_k-E_5$"', blue] (1.5, 1)
	(1.5, 1) edge ["$L-E_1-E_k$"'] (1.5,2);
\end{tikzpicture}
\begin{tikzpicture}
[
	>=stealth,
	point/.style = {draw, circle,  fill = black, inner sep = 1pt},
	every edge quotes/.style = {font=\tiny, color = black, auto=left, inner sep = 1pt},
]	
	\draw node at (-1.5,1) [point]{};
	\draw node at (-1,1) [point]{};
	\draw node at (1,1) [point]{};
	\draw node at (1.5,0.75) [point]{};
	\fill (0,0) ellipse [x radius=50pt, y radius=4pt] node[below=2pt]{\tiny $L-E_j-E_k-E_5$};
	\fill (0,2) ellipse [x radius=50pt, y radius=4pt] node[above=2pt]{\tiny $E_1-E_i$};
	\draw
	(-1.5,0) edge ["$L-E_1-E_i$"] (-1.5, 1)
	(-1.5, 1) edge ["$E_i$"] (-1.5,2);
	\draw
	(-1,0) edge ["$E_j$"'] (-1, 1)
	(-1, 1) edge ["$L-E_1-E_j$"'] (-1,2);
	\draw
	(1,0) edge ["$E_k$"] (1, 1)
	(1,1) edge ["$L-E_1-E_k$" near start] (1,2);
	\draw
	(1.5,0) edge ["$E_5$"'] (1.5, 0.75)
	(1.5, 0.75) edge ["$L-E_1-E_5$"'] (1.5,2);
\end{tikzpicture}
\caption{Blowup ($E_5$) on the minimal surface of the left graph of Figure \ref{blowup3-4} where $\{i, j, k\} = \{2, 3, 4\}$. If the next blowup ($E_6$) is performed on the minimal surface, it produces $L-E_j-E_k-E_5-E_6$. If the next blowup ($E_6$) is performed on the maximal surface, it produces $E_1-E_i-E_6$. If the next blowup ($E_6$) is performed at an interior fixed point, the minimal surface remains as $L-E_j-E_k-E_5$. 
}\labell{blowup4-4Lc}
\end{figure}

\begin{figure}[h]
\begin{tikzpicture}
[
	>=stealth,
	point/.style = {draw, circle,  fill = black, inner sep = 1pt},
	every edge quotes/.style = {font=\tiny, color = black, auto=left, inner sep = 1pt},
]	
	\draw node at (-1,1) [point]{};
	\draw node at (0,1) [point]{};
	\draw node at (1,1) [point]{};
	\draw node at (1,1.25) [point]{};
	\fill (0,0) ellipse [x radius=40pt, y radius=4pt] node[below=2pt]{\tiny $L-E_2-E_3-E_4$};
	\fill (0,2) ellipse [x radius=40pt, y radius=4pt] node[above=2pt]{\tiny $E_1-E_5$};
	\draw (-1,0) -- 
	(-1, 1) -- 
	(-1,2);
	\draw (0,0) -- 
	(0, 1) -- 
	(0,2);
	\draw (1,0) edge ["$E_k$"'] (1, 1)
	(1, 1) edge ["$L-E_1-E_k-E_5$"', blue] (1, 1.25)
	(1, 1.25) edge ["$E_5$"'] (1,2);
\end{tikzpicture}
\begin{tikzpicture}
[
	>=stealth,
	point/.style = {draw, circle,  fill = black, inner sep = 1pt},
	every edge quotes/.style = {font=\tiny, color = black, auto=left, inner sep = 1pt},
]	
	\draw node at (-1,1) [point]{};
	\draw node at (-0.5,1) [point]{};
	\draw node at (0.5,1) [point]{};
	\draw node at (1,1.25) [point]{};
	\fill (0,0) ellipse [x radius=40pt, y radius=4pt] node[below=2pt]{\tiny $L-E_2-E_3-E_4$};
	\fill (0,2) ellipse [x radius=40pt, y radius=4pt] node[above=2pt]{\tiny $E_1-E_5$};
	\draw (-1,0) -- (-1, 1)
	-- (-1,2);
	\draw (-0.5,0) -- 
	(-0.5, 1) -- 
	(-0.5,2);
	\draw (0.5,0) -- 
	(0.5, 1) -- 
	(0.5,2);
	\draw (1,0) edge ["$L-E_1-E_5$"'] (1, 1.25)
	(1, 1.25) edge ["$E_5$"'] (1,2);
\end{tikzpicture}
\caption{Blowup ($E_5$) on the maximal surface of the right graph of Figure \ref{blowup3-4} where $k \in \{2, 3, 4\}$. If the next blowup ($E_6$) is performed on the maximal surface, it produces $E_1-E_5-E_6$. If the next blowup ($E_6$) is performed on the minimal surface, it produces $L-E_2-E_3-E_4-E_6$. If the next blowup ($E_6$) is performed at an interior fixed point, the minimal surface remains as $L-E_2-E_3-E_4$.
}\labell{blowup4-4Ra}
\end{figure}
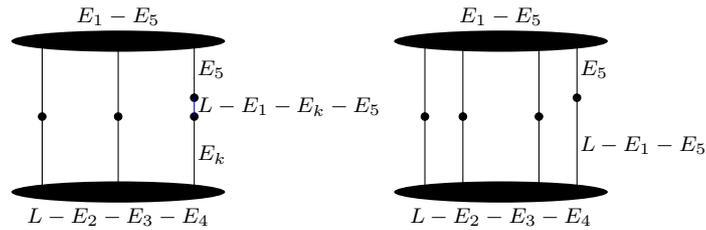


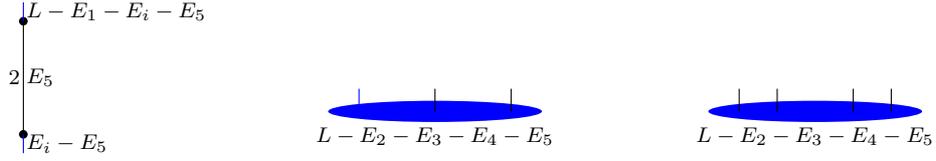
\begin{figure}[h]
\begin{tikzpicture}
[
	>=stealth,
	point/.style = {draw, circle,  fill = black, inner sep = 1pt},
	every edge quotes/.style = {font=\tiny, color = black, auto=left, inner sep = 1pt},
]	
	\draw node at (0,0.25) [point]{};
	\draw node at (0,1.75) [point]{};
	\draw
	(0,0) edge ["$E_i-E_5$"', blue] (0, 0.25)
	(0,0.25) edge ["$E_5$"', "$2$"] (0, 1.75)
	(0,1.75) edge ["$L-E_1-E_i-E_5$"', blue] (0,2);
\end{tikzpicture}
\hspace{1cm}
\begin{tikzpicture}
[
	>=stealth,
	point/.style = {draw, circle,  fill = black, inner sep = 1pt},
]	
	\fill[blue] (0,0) ellipse [x radius=40pt, y radius=4pt] node[below=2pt]{\color{black}\tiny $L-E_2-E_3-E_4-E_5$};
	\draw[blue] (-1, 0) -- (-1,0.3);
	\draw (0, 0) -- (0,0.3);
	\draw (1, 0) -- (1,0.3);

	\fill[blue] (5,0) ellipse [x radius=40pt, y radius=4pt] node[below=2pt]{\color{black}\tiny $L-E_2-E_3-E_4-E_5$};
	\draw (4,0) -- (4,0.3);
	\draw (4.5, 0) -- (4.5,0.3);
	\draw (5.5, 0) -- (5.5,0.3);	
	\draw (6, 0) -- (6,0.3);
\end{tikzpicture}	
\caption{If the blowup ($E_5$) is performed at an interior fixed point of the right graph of Figure \ref{blowup3-4} where $i \in \{2, 3, 4\}$, then $E_5$ remains after the next blowup $(E_6)$ due to size conditions. If the blowup ($E_5$) is performed on the minimal surface of the right graph of Figure \ref{blowup3-4}, then $L-E_2-E_3-E_4-E_5$ remains after the next blowup ($E_6$) due to size conditions.}\labell{blowup4-4Rb}
\end{figure}

\vfill
\pagebreak
\end{noTitle}

\begin{Remark}\labell{ex1b}
There is another homologically trivial $\Z_{2}$-action that does not extend to a Hamiltonian circle action on $M_6$ with a blowup form in the cohomology class $\tilde{\Omega}$: start from the Hamiltonian $S^1$-space with the decorated graph on the right of Figure \ref{blowup3-4}, perform an $S^1$-equivariant blowup at an interior fixed point to get a $\Z_{2}$-sphere in $E_5$; and then perform a $\Z_{2}$-equivariant complex blowup at a point in this $\Z_{2}$-sphere that is not an $S^1$-fixed point; the resulting complex manifold is again a weak del Pezzo surface of degree $3$ and type $A_4$; its $(-2)$- and $(-1)$-curves are in the classes
$L-E_1-E_4-E_5$, $E_5-E_6$, $E_4-E_5$, $L-E_2-E_3-E_4$, $E_6$, $E_1$, $E_2$, $E_3$, $L-E_1-E_2$, $L-E_1-E_3$, with embedded $\Z_2$-fixed spheres in $E_1$, $E_5-E_6$ and $L-E_2-E_3-E_4$ (see Figure \ref{config-ex1b}); the rest is done by similar arguments.
\end{Remark}

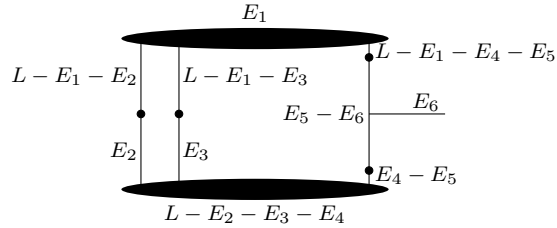
\begin{figure}[h]
\begin{tikzpicture}
[
	point/.style = {draw, circle,  fill = black, inner sep = 1pt},
	every edge quotes/.style = {font=\tiny, color = black, auto=left, inner sep = 1pt},
	loosestart/.style = {inner sep= 2pt, near start},
	looseend/.style = {inner sep= 2pt, near end},
]	 	
	\fill (0,0) ellipse [x radius=50pt, y radius=4pt] node[below=2pt]{\tiny $L-E_2-E_3-E_4$};
	\fill (0,2) ellipse [x radius=50pt, y radius=4pt] node[above=2pt]{\tiny $E_1$};	
	\draw node at (-1.5,1) [point]{};
	\draw node at (-1,1) [point]{};
	\draw node at (1.5,0.25) [point]{};
	\draw node at (1.5,1.75) [point]{};
	\draw
	(-1.5,0) edge ["$E_2$"] (-1.5,1)
	(-1.5,1) edge ["$L-E_1-E_2$"] (-1.5,2);
	\draw
	(-1,0) edge ["$E_3$"'] (-1,1)
	(-1,1) edge ["$L-E_1-E_3$"'] (-1,2);
	\draw
	(1.5,0) edge ["$E_4-E_5$"' looseend] (1.5,0.25)
	(1.5,0.25) edge ["$E_5-E_6$"] (1.5,1.75)
	(1.5,1.75) edge ["$L-E_1-E_4-E_5$"' loosestart] (1.5,2);
	\draw (1.5,1) edge ["$E_6$" near end] (2.5,1);
\end{tikzpicture}	
\caption{Configuration for Remark \ref{ex1b}.}\labell{config-ex1b}
\end{figure}

\section{The example in the non-simply connected case} \label{ex-nonsimply}

\begin{noTitle}\labell{ex-nonsc}
\textbf{We construct a ${\Z}_{2}$-action on $W_{3}$ with a blowup form $\omega_{1,1;\frac{3}{5},\frac{7}{20},\frac{3}{10}}$} as follows:
Start from the Hamiltonian $S^1$-action on $(\Sigma \times S^2,\omega_{1,1})$ whose decorated graph is in Figure \ref{graph-ruled} with $\ell=0$.
Perform an $S^1$-equivariant symplectic blowup of size $\frac{3}{5}$ at a point in the minimal fixed surface, followed by an $S^1$-equivariant blowup of size $\frac{7}{20}$ at the obtained interior fixed point, to get a ${\Z}_{2}$-sphere in $E_2$. 
We get a Hamiltonian circle action on $M=W_{2}$ with a blowup form $\omega=\omega_{1,1;\frac{3}{5}, \frac{7}{20}}$, whose decorated graph is shown on the left of Figure \ref{fig-genusp-action}.
\begin{figure}[h]
\begin{tikzpicture}
[
	>=stealth,
	point/.style = {draw, circle,  fill = black, inner sep = 1pt},
	every edge quotes/.style = {font=\tiny, color = black, auto=left, inner sep = 1pt},
	loosestart/.style = {inner sep= 5pt, near start},
]	
	\draw node at (-0.5,0.5) [point]{};
	\draw node at (-0.5,1.9) [point]{};
	\fill (0,0) ellipse [x radius=20pt, y radius=2pt] node[below=2pt]{\color{black}\tiny $B-E_1$};
	\fill (0,2) ellipse [x radius=20pt, y radius=2pt] node[above=2pt]{\color{black}\tiny $B$};
	\draw
	(-0.5,0) edge ["$E_1-E_2$", blue] (-0.5, 0.5)
	(-0.5, 0.5) edge ["$E_2$", "$2$"'] (-0.5,1.9)
	(-0.5, 1.9) edge ["$F-E_1-E_2$" loosestart, blue] (-0.5,2);
	\draw (0.5,2) edge ["$F$"] (0.5,0);
\end{tikzpicture}
\hspace{1cm}
\begin{tikzpicture}
[
	>=stealth,
	point/.style = {draw, circle,  fill = black, inner sep = 1pt},
	every edge quotes/.style = {font=\tiny, color = black, auto=left, inner sep = 1pt},
	loosestart/.style = {inner sep= 5pt, near start},
]	
	\draw node at (-0.5,0.5) [point]{};
	\draw node at (-0.5,1.9) [point]{};
	\draw node at (0.5,1.4) [point]{};
	\fill (0,0) ellipse [x radius=20pt, y radius=2pt] node[below=2pt]{\color{black}\tiny $B-E_1$};
	\fill (0,2) ellipse [x radius=20pt, y radius=2pt] node[above=2pt]{\color{black}\tiny $B-E_3$};
	\draw
	(-0.5,0) edge ["$E_1-E_2$"] (-0.5, 0.5)
	(-0.5, 0.5) edge ["$E_2$", "$2$"'] (-0.5,1.9)
	(-0.5, 1.9) edge ["$F-E_1-E_2$" loosestart] (-0.5,2);
	\draw (0.5,2) edge ["$E_3$"] (0.5,1.4)
	(0.5, 1.4) edge ["$F-E_3$"] (0.5,0);
\end{tikzpicture}
\caption{Special decorated graphs.}\labell{fig-genusp-action}
\end{figure}
By \cite[Theorem 7.1]{karshon}, $(M,\omega)$ admits an integrable complex structure  
such that the $S^1$-action is holomorphic and the symplectic form is K\"ahler. 
Now perform a ${\Z}_{2}$-equivariant complex blowup $$\pi \colon \tilde{M} \to M$$ at a point $p$ in the ${\Z}_{2}$-sphere in $E_2$ that is not an $S^1$-fixed point. 
Let $\tilde{E_3}=\pi^{-1}(p)$ be the exceptional divisor and $\Xi_{3}$ be its Poincar\'e dual in $H^{2}(\tilde{M})$. Let
$$
\tilde{\Omega}:=\pi^{*}\Omega-\frac{3}{10} \Xi_{3},
$$
where $\Omega$ is the cohomology class of $\omega$.
We claim that $\tilde{\Omega}$ contains a K\"ahler form. 
As before,
by Nakai's criterion, and since $\tilde{\Omega}^2 
> 0$, it is enough to show that $\la \tilde{\Omega},[C]\ra > 0$ for every complex curve $C \subset \tilde{M}$. This follows from the following claim.

\begin{Claim} \labell{generators}
Consider the classes 
\[
	F-E_1-E_2,\ E_2-E_3,\ E_3,\ E_1-E_2,\ B-E_1.
	\] 
	\begin{enumeratei}
	\item \labell{part1} For each of these classes, its pairing with $\tilde{\Omega}$ is positive. 
	\item \labell{part2} Each of the classes contains a ${\Z}_{2}$-invariant complex curve. 
		\item \labell{part3} The classes generate the effective cone of complex curves in $\tilde{M}$.
       \end{enumeratei}
\end{Claim}

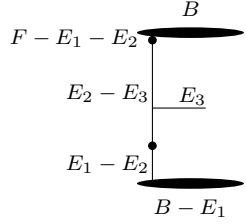
\begin{figure}[h]
\begin{tikzpicture}
[
	>=stealth,
	point/.style = {draw, circle,  fill = black, inner sep = 1pt},
	every edge quotes/.style = {font=\tiny, color = black, auto=left, inner sep = 1pt},
	loosestart/.style = {inner sep= 5pt, near start},
]	
	\draw node at (-0.5,0.5) [point]{};
	\draw node at (-0.5,1.9) [point]{};
	\fill (0,0) ellipse [x radius=20pt, y radius=2pt] node[below=2pt]{\color{black}\tiny $B-E_1$};
	\fill (0,2) ellipse [x radius=20pt, y radius=2pt] node[above=2pt]{\color{black}\tiny $B$};
	\draw
	(-0.5,0) edge ["$E_1-E_2$"] (-0.5, 0.5)
	(-0.5, 0.5) edge ["$E_2-E_3$"] (-0.5,1.9)
	(-0.5, 1.9) edge ["$F-E_1-E_2$" loosestart] (-0.5,2);
	\draw (-0.5,1) edge ["$E_3$" near end] (0.2,1);
\end{tikzpicture}
\caption{Special configuration.}\labell{countex2}
\end{figure}

\begin{proof}
\itemref{part1} follows from the definition of $\tilde{\Omega}$. \itemref{part2} follows from the blowup construction.
For \itemref{part3}, we only need to verify that the proper transform of an irreducible curve on $M$, that passes through the point $p$ with multiplicity $m$, represents a non-negative combination of these classes.  
Recall that the classes
\[
F-E_1-E_2,\ E_2,\ E_1-E_2,\ B-E_1.
\]
generate the effective cone in $M$ (before the third blowup), and each is represented by an $S^1$-invariant curve.
This follows from \cite[Proposition C.6 and 
Theorem 7.1]{karshon}, or an inductive argument similar to the computation below. 
Hence the proper transform can be expressed as 
\[ 
\begin{split}
\tilde{D} &= D - mE_3 \\
&= a(F-E_1-E_2)+b E_2+c (E_1-E_2)+d (B-E_1)-mE_3 \\
&=a(F-E_1-E_2)+b (E_2-E_3)+(b-m)E_3+ c (E_1-E_2)+d (B-E_1)
\end{split}
\] with $a, b, c, d \geq 0$. 
It remains to show that $b-m \geq 0$.   

Denote by $\pi_1$ and $\pi_2$ the first and second blowup maps at $p_1$ and $p_2$ in the construction of $\tilde{M}$. We have  
\[
\begin{split}
m = \tilde{D} \cdot E_3 &= \mult_{p} D \\
&\leq D \cdot E_2 = \mult_{p_2} \pi_2(D) \\
&\leq \pi_2(D)\cdot E_1 = \mult_{p_1} \pi_1(\pi_2(D))\\ 
&\leq \pi_1\pi_2(D)\cdot B  = (aF+dB) \cdot B \leq a.
\end{split}
\]
Also, by positivity of intersections of complex curves in $M$,
\[
0 \leq D \cdot (F-E_1-E_2) = -2a + b. 
\]
Combining these two inequalities, we have $m \leq b$.
\end{proof}

The configuration of the generating classes is illustrated in Figure \ref{countex2}. Again note that it is not a decorated graph for a Hamiltonian circle action.

By averaging with respect to the holomorphic ${\Z}_{2}$-action, we get an invariant K\"ahler form $\tilde{\omega}$ on $\tilde{M}$ in $\tilde{\Omega}$. 
By \cite[Lemma 3.5 Part 2]{Li-Liu} and \cite[Theorem A]{liliu-ruled}, since the Chern class of $\tilde{\omega}$ is the same as that of a blowup form, it is a blowup form. 
This yields a homologically trivial symplectic ${\Z}_{2}$-action on $(\tilde{M},\tilde{\omega})=(W_{3}, \omega_{1,1;\frac{3}{5},\frac{7}{20},\frac{3}{10}})$. 
The adjunction formula implies that the ${\Z}_{2}$-fixed sphere in $E_2-E_3$ is embedded.

Note that the manifold $(W_{3}, \omega_{1,1;\frac{3}{5},\frac{7}{20},\frac{3}{10}})$ does admit a Hamiltonian circle action, whose decorated graph is shown on the right of Figure \ref{fig-genusp-action}.

\end{noTitle}

\begin{noTitle}
\textbf{We claim that there is no Hamiltonian circle action that extends the constructed ${\Z}_{2}$-action.}
Note that the vector $(1,1;\frac{3}{5},\frac{7}{20},\frac{3}{10})$ is reduced. By Theorem \ref{charpositive},
 an extending Hamiltonian  circle action would have been obtained from a Hamiltonian circle action on $(\Sigma \times S^2,\omega_{1,1})$ by $S^1$-equivariant symplectic blowups of sizes  ${3}/{5},{7}/{20},{3}/{10}$, in this order. 
There is exactly one Hamiltonian $S^1$-action on $(\Sigma \times S^2,\omega_{1,1})$, whose decorated graph is shown in Figure \ref{graph-ruled} with $\ell=0$. 
Therefore, a Hamiltonian $S^1$-action on 
 $(W_{3}, \omega_{1,1;\frac{3}{5},\frac{7}{20},\frac{3}{10}})$
 is obtained from this action by one of the following sequences of equivariant blowups of sizes ${3}/{5},{7}/{20},{3}/{10}$ (up to a change of compatible metric (or almost complex structure)).

\begin{description}
\item[Type I] two blowups at one fixed surface and one at the other (performed simultaneously); 
\item[Type II] a blowup at a fixed surface, followed by a second blowup at the interior fixed point created at the first blowup (to get a ${\Z}_{2}$-sphere in $E_2$), and a third blowup at one of the fixed surfaces;
\item[Type III] the first and the second blowups are each at a fixed surface (either the same one, or different ones), followed by a blowup at the interior fixed point created at the first blowup (note that the second and the third blowups can be performed simultaneously);
\item[Type IV] the first and the second blowups are each at a fixed surface (either the same one, or different ones), followed by a blowup at the interior fixed point created at the second blowup.
\end{description}

The other options are ruled out due to size restrictions. In particular, if the second blowup is at the interior fixed point created at the first blowup, obtaining a ${\Z}_{2}$-sphere in $E_2$, then one cannot perform an $S^1$-blowup of size $\frac{3}{10}$ at one of its poles, as shown on the left of Figure \ref{fig-genusp-action} (by the blue edges).

By the proof of \cite[Theorem 7.1]{karshon}, for each type of the Hamiltonian $S^1$-actions, the complex structure $\overline{J}$ obtained from the product complex structure on $\Sigma \times S^2$ by $S^1$-equivariant complex blowups at the same points, in the same order, is compatible with the symplectic form. 
Note that an exceptional divisor is a $\overline{J}$-holomorphic sphere, and so is the proper transform of the exceptional divisor of a previous blowup. By Lemma \ref{clsta}, a ${\Z}_{2}$-sphere is also $\overline{J}$-holomorphic. In particular, we will have 
\begin{itemize}
	\item a $\overline{J}$-holomorphic sphere in $E_2$ in Types I, II, and III; 
	\item a  $\overline{J}$-holomorphic sphere in $E_2-E_3$ that is not fixed by ${\Z}_{2}$ 	in Type IV (as shown in Figure \ref{type4} or its upside-down). 
\end{itemize}

\begin{figure}[h]
\begin{tikzpicture}
[
	>=stealth,
	point/.style = {draw, circle,  fill = black, inner sep = 1pt},
	every edge quotes/.style = {font=\tiny, color = black, auto=left, inner sep = 1pt},
	loosestart/.style = {inner sep= 5pt, near start},
]	
	\draw node at (-0.5,1.2) [point]{};
	\draw node at (0.5,0.1) [point]{};
	\draw node at (0.5,1.3) [point]{};
	\fill (0,0) ellipse [x radius=20pt, y radius=2pt] node[below=2pt]{\color{black}\tiny $B-E_1-E_2$};
	\fill (0,2) ellipse [x radius=20pt, y radius=2pt] node[above=2pt]{\color{black}\tiny $B$};
	\draw
	(-0.5,0) edge ["$E_1$"] (-0.5, 1.2)
	(-0.5, 1.2) edge ["$F-E_1$"] (-0.5,2);
	\draw 
	(0.5,2) edge ["$F-E_2-E_3$"] (0.5,1.3)
	(0.5,1.3) edge ["$E_3$", "$2$"'] (0.5,0.1)
	(0.5, 0.1) edge ["$E_2-E_3$" loosestart] (0.5,0);
\end{tikzpicture}
\hspace{1cm}
\begin{tikzpicture}
[
	>=stealth,
	point/.style = {draw, circle,  fill = black, inner sep = 1pt},
	every edge quotes/.style = {font=\tiny, color = black, auto=left, inner sep = 1pt},
	loosestart/.style = {inner sep= 5pt, near start},
]	
	\draw node at (-0.5,1.2) [point]{};
	\draw node at (0.5,0.7) [point]{};
	\draw node at (0.5,1.9) [point]{};
	\fill (0,0) ellipse [x radius=20pt, y radius=2pt] node[below=2pt]{\color{black}\tiny $B-E_1$};
	\fill (0,2) ellipse [x radius=20pt, y radius=2pt] node[above=2pt]{\color{black}\tiny $B-E_2$};
	\draw
	(-0.5,0) edge ["$E_1$"] (-0.5, 1.2)
	(-0.5, 1.2) edge ["$F-E_1$"] (-0.5,2);
	\draw 
	(0.5,2) edge ["$E_2-E_3$" loosestart] (0.5,1.9)
	(0.5, 1.9) edge ["$E_3$", "$2$"'] (0.5,0.7)
	(0.5, 0.7) edge ["$F-E_2-E_3$"] (0.5,0);
\end{tikzpicture}
\caption{Decorated graphs of Type IV.}\labell{type4}
\end{figure}
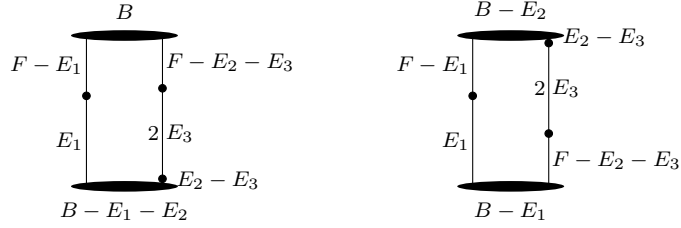

 However, by the construction of the ${\Z}_{2}$-action and the assumption that the Hamiltonian circle action extends the ${\Z}_{2}$-action, there is a ${\Z}_{2}$-fixed embedded sphere in $E_2-E_3$.
 By Lemma \ref{clsta}, this sphere is  $\overline{J}$-holomorphic. 
 Its intersection with a different $\overline{J}$-holomorphic sphere in $E_2-E_3$ (e.g. one that is not ${\Z}_{2}$-fixed), or with a $\overline{J}$-holomorphic sphere in $E_2$ is negative. Therefore, an extending Hamiltonian circle action of each of the four types would have yielded a contradiction to positivity of intersections of holomorphic curves.

\end{noTitle}

\begin{Remark}
The above construction can be generalized to produce a symplectic $\Z_r$-action on $W_{r+1}$ 
with a blowup form, that does not extend to a Hamiltonian circle action, for an even integer $r > 2$.
Start from the same Hamiltonian $S^1$-action on $(\Sigma \times S^2,\omega_{1,1})$ as before. 
Perform a sequence of $r$ $S^1$-equivariant symplectic blowups so that the
resulting decorated graph is as the one illustrated on the left of Figure \ref{fig-genusp-2} for $r=4$.  
Next perform a ${\Z}_{r}$-equivariant complex blowup at a point in the $\Z_r$-sphere in $E_r$ that is not an $S^1$-fixed point.
The resulting configuration of $\Z_r$-invariant complex curves on $W_{r+1}$ is illustrated in Figure \ref{countex2-2} for $r=4$. 
Assume that the blowup sizes $\delta_1,\ldots,\delta_r$ and $\delta_{r+1}$
 satisfy certain conditions; for example, we take $\delta_i= \frac{2^{2r-i}+1}{2^{2r}} \text{ for }1 \leq i \leq r$, and $\delta_{r+1}=\frac{1}{2^{r}}$.   
 Since $r$ is even, $\Z_2 < \Z_r$. Similar arguments as before show that we obtain a $\Z_r$-invariant blowup form $\omega=\ow_{1,1;\delta_1, \dots, \delta_{r+1}}$ on $W_{r+1}$, and that 
 for any invariant $\omega$-compatible almost complex structure $J$
 there are $J$-holomorphic spheres, $\Z_r$-fixed in $E_r-E_{r+1}$ and $\Z_2$-fixed in $E_{2i}-E_{2i+1}$ for $1 \leq i < \frac{r}{2}$.
The symplectic manifold $(W_{r+1},\omega)$ does admit a Hamiltonian $S^1$-action, as shown on the right of Figure \ref{fig-genusp-2} for $r=4$.
By a  careful analysis on the $J$-holomorphic spheres that can occur under a Hamiltonian $S^1$-action, we can show that an extending Hamiltonian $S^1$-action would lead to a violation of positivity of intersections.

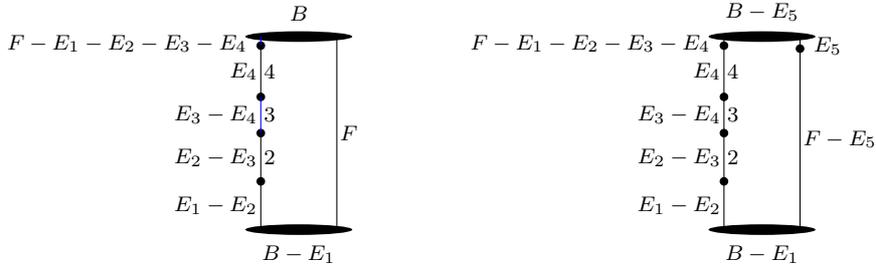
\begin{figure}[h]
\begin{tikzpicture}
[
	>=stealth,
	point/.style = {draw, circle,  fill = black, inner sep = 1pt},
	every edge quotes/.style = {font=\tiny, color = black, auto=left, inner sep = 1pt},
	loosestart/.style = {inner sep= 5pt, near start},
]	
	\draw node at (-0.5,0.64) [point]{};
	\draw node at (-0.5,1.28) [point]{};
	\draw node at (-0.5,1.76) [point]{};
	\draw node at (-0.5,2.44) [point]{};
	\fill (0,0) ellipse [x radius=20pt, y radius=2pt] node[below=2pt]{\color{black}\tiny $B-E_1$};
	\fill (0,2.56) ellipse [x radius=20pt, y radius=2pt] node[above=2pt]{\color{black}\tiny $B$};
	\draw
	(-0.5,0) edge ["$E_1-E_2$"] (-0.5, 0.64)
	(-0.5, 0.64) edge ["$E_2-E_3$", "$2$"'] (-0.5,1.28)
	(-0.5, 1.28) edge ["$E_3-E_4$", "$3$"', blue] (-0.5,1.76)
	(-0.5, 1.76) edge ["$E_4$", "$4$"'] (-0.5,2.44)
	(-0.5, 2.44) edge ["$F-E_1-E_2-E_3-E_4$" loosestart, blue] (-0.5,2.56);
	\draw (0.5,2.56) edge ["$F$"] (0.5,0);
\end{tikzpicture}
\hspace{1cm}
\begin{tikzpicture}
[
	>=stealth,
	point/.style = {draw, circle,  fill = black, inner sep = 1pt},
	every edge quotes/.style = {font=\tiny, color = black, auto=left, inner sep = 1pt},
	loosestart/.style = {inner sep= 5pt, near start},
	looseend/.style = {inner sep= 5pt, near end},
]	
	\draw node at (-0.5,0.64) [point]{};
	\draw node at (-0.5,1.28) [point]{};
	\draw node at (-0.5,1.76) [point]{};
	\draw node at (-0.5,2.44) [point]{};
	\draw node at (0.5,2.4) [point]{};
	\fill (0,0) ellipse [x radius=20pt, y radius=2pt] node[below=2pt]{\color{black}\tiny $B-E_1$};
	\fill (0,2.56) ellipse [x radius=20pt, y radius=2pt] node[above=2pt]{\color{black}\tiny $B-E_5$};
	\draw
	(-0.5,0) edge ["$E_1-E_2$"] (-0.5, 0.64)
	(-0.5, 0.64) edge ["$E_2-E_3$", "$2$"'] (-0.5,1.28)
	(-0.5, 1.28) edge ["$E_3-E_4$", "$3$"'] (-0.5,1.76)
	(-0.5, 1.76) edge ["$E_4$", "$4$"'] (-0.5,2.44)
	(-0.5, 2.44) edge ["$F-E_1-E_2-E_3-E_4$" loosestart] (-0.5,2.56);
	\draw (0.5,2.56) edge ["$E_5$" looseend] (0.5,2.4)
	(0.5, 2.4) edge ["$F-E_5$"] (0.5,0);
\end{tikzpicture}
\caption{Special decorated graphs on $W_4$ and $W_5$.}\labell{fig-genusp-2}
\end{figure}

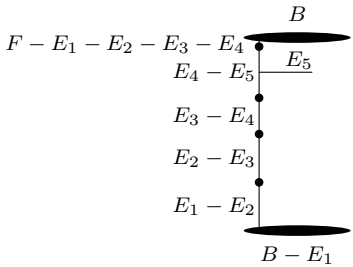
\begin{figure}[h]
\begin{tikzpicture}
[
	>=stealth,
	point/.style = {draw, circle,  fill = black, inner sep = 1pt},
	every edge quotes/.style = {font=\tiny, color = black, auto=left, inner sep = 1pt},
	loosestart/.style = {inner sep= 5pt, near start},
]	
	\draw node at (-0.5,0.64) [point]{};
	\draw node at (-0.5,1.28) [point]{};
	\draw node at (-0.5,1.76) [point]{};
	\draw node at (-0.5,2.44) [point]{};
	\fill (0,0) ellipse [x radius=20pt, y radius=2pt] node[below=2pt]{\color{black}\tiny $B-E_1$};
	\fill (0,2.56) ellipse [x radius=20pt, y radius=2pt] node[above=2pt]{\color{black}\tiny $B$};
	\draw
	(-0.5,0) edge ["$E_1-E_2$"] (-0.5, 0.64)
	(-0.5, 0.64) edge ["$E_2-E_3$"] (-0.5,1.28)
	(-0.5, 1.28) edge ["$E_3-E_4$"] (-0.5,1.76)
	(-0.5, 1.76) edge ["$E_4-E_5$"] (-0.5,2.44)
	(-0.5, 2.44) edge ["$F-E_1-E_2-E_3-E_4$" loosestart] (-0.5,2.56);	
	\draw (-0.5,2.1) edge ["$E_5$" near end] (0.2,2.1);
\end{tikzpicture}
\caption{Special configuration on $W_5$.}\labell{countex2-2}
\end{figure}

By contrast, a generalization for the simply connected case is more involved. 
The main difficulty is to determine the generators of the effective cone of complex curves.  Note that the effective cone may not be finitely generated for a complex blowup of $\CP^2$ at $9$ points or more.

\end{Remark}

\section{Further Remarks}\label{remarks}

A symplectic action of a finite cyclic group $\Z/\Z_n$ is a Hamiltonian action with the moment map being identically zero. However, this does not mean the action is generated by a Hamiltonian diffeomorphism. The generator of the action is a periodic map of period $n$ in $\Symp(M, \ow)$, or in the Torelli part of $\Symp(M, \ow)$ if the action is homologically trivial.

It would be very interesting to know whether any of our $\Z_2$-actions is generated by a Hamiltonian diffeomorphism and therefore corresponds to a $\Z_2$-subgroup in 
$\Ham(M, \ow)$.  
This is a hard question: we do not have an answer but we give two remarks concerning the example in Section \ref{ex-simply}.

\begin{Remark}
	Let $G = \Z/p\Z$ be a finite cyclic group of prime order $p$.
	Consider a $G$-action generated by a Hamiltonian diffeomorphism on a closed symplectic manifold $M$. 
	Let $F$ be the fixed point set of the $G$-action. Denote by $SB$ the sum of Betti numbers. Then  
	\[
	SB (F) = SB (M),
	\] while homology is taken with $\Z/p\Z$ coefficients. 
	
	Indeed, $SB(F)$ is bounded above by $SB(M)$ for any cyclic action of prime order $p$, not necessarily generated by a Hamiltonian diffeomorphism \cite[Chapter 3]{bredon}. On the other hand, let $\sigma$ denote the Hamiltonian diffeomorphism generating the cyclic action. The graph of $\sigma$ and the diagonal intersect cleanly in $M \times M$. Then Arnold's conjecture implies that $SB(F)$ is bounded below by $SB(M)$, using Floer homology for clean intersections \cite{pozniak}.   
	
	The $\Z_2$-action on $(\tM, \tow)$ constructed in Section \ref{example-construction} has $3$ fixed spheres and $3$ fixed points. Hence $SB(F) = 9 = SB(\tM = M_6)$. It does not conclude whether this $\Z_2$-action is generated by a Hamiltonian diffeomorphism.   
\end{Remark} 

\begin{Remark}
The generator of the $\Z_2$-action constructed in Section \ref{example-construction} defines a non-trivial symplectomorphism $\sigma \colon (\tM, \tow) \to (\tM, \tow)$, where $\tM=M_6$ and $\tow =\ow_{1;\frac{1}{2},\frac{1}{4},\frac{1}{4},\frac{1}{4},\frac{3}{16},\frac{1}{8}}$.
We claim that this map is smoothly isotopic to the identity. The idea goes as follows: For any $0<\delta \leq \frac{1}{8}$, one can perform a ${\Z}_{2}$-equivariant symplectic blowup of size $\delta$ at the same point $p$ in the ${\Z}_{2}$-sphere $S$ in $E_5$ in $M$. The induced ${\Z}_{2}$-action on the blowup $\tM$ defines a diffeomorphism (in fact a symplectomorphism with a different symplectic form) $\sigma_{\delta} \colon \tM \to \tM$, which is smoothly isotopic to $\sigma$. For $\delta$ sufficiently small, one finds an isotopy of equivariant embeddings of $\delta$-balls in the local normal form of the $S^1$-invariant sphere $S$. Hence $\sigma_{\delta}$ is isotopic to the generator $\sigma_{p_0}\colon \tM \to \tM$ of the ${\Z}_{2}$-action induced by a ${\Z}_{2}$-equivariant blowup of size $\delta$ at the north pole $p_0$ of $S$. If $\delta$ is sufficiently small, one can perform an $S^1$-equivariant symplectic blowup of size $\delta$ at the $S^1$-fixed point $p_0$ \cite[Theorem 7.1, Proposition 7.2]{karshon}, to get a Hamiltonian $S^1$-action on $\tM$; the element $\pi$ in $S^1$ defines a symplectomorphism $\sigma_{S^1}\colon \tM \to \tM$, which is smoothly isotopic to $\sigma_{p_0}$ and hence to $\sigma$. Since $S^1$ is connected, $\sigma_{S^1}$ is isotopic to the identity.

We do not know whether the symplectomorphism $\sigma$ is symplectically isotopic to the identity.

If it is, it would present a cyclic subgroup in 
$\Ham(\tM, \tow)$ which does not embed in a circle subgroup of $\Ham(\tM, \tow)$.
This is due to the fact that $\Ham(\tM, \tow)$ coincides with the identity component of $\Symp(\tM, \tow)$ when $H^1(\tM)$ is trivial.

If it is not, it may be symplectically isotopic to a Dehn--Seidel twist, or it may provide another example of a symplectomorphism that is smoothly, but not symplectically, isotopic to the identity.

\end{Remark}


\end{document}